\documentclass{amsart}
\usepackage{amssymb} 
\usepackage{amsmath} 
\usepackage{mathrsfs} 
\usepackage{eufrak}
\usepackage{amscd}
\usepackage{amsbsy}
\usepackage{comment}
\usepackage[matrix,arrow]{xy}
\usepackage{hyperref}
\usepackage{mathtools}
\usepackage{esint}
\usepackage[letterpaper, margin=1in]{geometry}
\usepackage{hyperref}
\usepackage{enumerate}

\newcommand{\N}{{\mathbb N}}

\newcommand{\R}{{\mathbb R}}

\newcommand{\sR}{{\mathcal R}}

\newcommand{\inv}{^{-1}}

\newcommand{\supp}{\text{supp}}
\newcommand{\ep}{\varepsilon}

\newcommand{\Rm}{\text{Rm}}
\newcommand{\Rc}{\text{Rc}}
\newcommand{\Vol}{\text{Vol}}
\newcommand{\inj}{\operatorname{inj}}
\newcommand{\AVR}{\operatorname{AVR}}
\newcommand{\RNum}[1]{\uppercase\expandafter{\romannumeral #1\relax}}

\newcommand{\Sob}{\operatorname{Sob}}
\let\epsilon\varepsilon

\newcommand{\be}{\begin{equation}}
\newcommand{\ee}{\end{equation}}
\newcommand{\ba}{\begin{align*}}

\newtheorem{thm}{Theorem}[section]
\newtheorem{lem}[thm]{Lemma}
\newtheorem{claim}[thm]{Claim}
\newtheorem{prop}[thm]{Proposition}
\newtheorem{cor}[thm]{Corollary}

\theoremstyle{definition}
\newtheorem{remark}[thm]{Remark}

\newtheorem{question}[thm]{Question}

\def\mod#1{{\ifmmode\text{\rm\ (mod~$#1$)}
\else\discretionary{}{}{\hbox{ }}\rm(mod~$#1$)\fi}}
\begin{document}

\setlength\parindent{0pt}

\vspace{-3em}

        \author[A. Martens]{Adam Martens}
        \address{Department of Mathematics, The University of British Columbia,
1984 Mathematics Road, Vancouver, B.C.,  Canada V6T 1Z2}
\email{martens@math.ubc.ca}

        \bibliographystyle{amsplain}
\title{Sharpening a gap theorem: nonnegative Ricci and small curvature concentration}

\maketitle

\pagestyle{plain}

\vspace{-3em}

\begin{abstract}
We sharpen a gap theorem of Chan \& Lee \cite{ChanLee} for nonnegative Ricci curvature manifolds that have positive asymptotic volume ratio and small enough scale-invariant integral curvature (so-called ``curvature concentration"), by showing that the curvature concentration need only depend linearly on the asymptotic volume ratio. We prove the result by exhibiting a long-time Ricci flow solution with faster than $1/t$ curvature decay, which allows us to shift the limiting contradiction argument to time infinity and thus obtain an explicit bound on the size of the gap. 
\end{abstract}

\section{Introduction}

The Ricci flow, introduced by Hamilton \cite{Hamilton1}, deforms an initial Riemannian metric $g(0)=g_0$ on a $n$-dimensional manifold $M^n$ according to the equation
\[
\frac{\partial }{\partial t}g(t)=-2\Rc_{g(t)}.
\]

A central aspect of Ricci flow analysis is the tendency for the flow to roughly preserve the geometry of almost Euclidean regions, despite the manifold's global geometry being possibly not well understood. This is the so-called ``pseudolocality" of the Ricci flow. Although the sense in which a region is ``almost Euclidean" may differ, the general arguments of pseudolocality theorems often follow this outline:

\begin{enumerate}[1.]
\item Construct a sequence of Ricci flows which are becoming more ``locally Euclidean" but for which some undesirable behaviour occurs (often of the form of large curvature).
\item Normalize the metrics to focus on the local region where the undesirable behavior occurs. 
\item Analyze the sequence of normalized metrics as they approach a limiting space, wherein a contradiction is apparent. 
\end{enumerate} 

We will refer to this method of proof (where a contradiction is derived in the limiting space) as a ``limiting contradiction argument". A notable aspect of such arguments is that they often involve constants that are not explicitly computable due to the nature of the contradiction appearing in the limit space instead of at a specific moment along the sequence. 

\vspace{5pt}

One of the reasons why pseudolocality theorems are important is that they relay geometric and topological information in the local region. In particular, if the sense in which a manifold $M$ is ``almost Euclidean" is invariant under scaling and is satisfied everywhere on $M$, then one can apply the result to a sequence of blow-downs of the initial metric, thus providing information about the manifolds global geometric structure. Consider for example the following notable result of Chan \& Lee.

\vspace{5pt}

\begin{thm}{\cite[Theorem 1.1]{ChanLee}}\label{ChanLeeGap}
For any $n\geq 4$ and $v>0$ there exists some $\ep$, depending on $n$ and $v$, such that if $(M^n,g)$ is complete with nonnegative Ricci curvature $\Rc_{g}\geq 0$, asymptotic volume ratio satisfying
\[
\lim_{r\to \infty} \frac{\Vol_g B_g(x_0, r)}{\omega_n r^n}\geq v,
\]
and globally small scale-invariant integral curvature (curvature concentration)
\be\label{chanleesmallcurv}
\int_M |\text{Rm}|_g^{\frac{n}{2}}\,dV_g\leq \ep(n,v),
\ee
then $(M^n,g)$ is isometric to $(\R^n, g_E)$. 
\end{thm}

\vspace{5pt}

The proof given in \cite{ChanLee} is a limiting contradiction argument wherein they fix some $v>0$ and then assume there is a sequence of nonnegative Ricci curved $(M_i^n,g_i)$, each with asymptotic volume ratio at least $v$ and integral curvatures satisfying \eqref{chanleesmallcurv} with $\ep_i\to 0$, which results in a contradiction in the limit. This raises two questions: 1. For any fixed $v>0$, precisely how small does $\ep(v)$ need to be to lie within the desired gap? 2. Given two distinct $v_1, v_2>0$, how does $\ep(v_1)$ relate to $\ep(v_2)$? The purpose of this paper is to provide an answer to both of these questions. We prove the following. 

\vspace{5pt}

\begin{thm}\label{main}
For $n\geq 4$, there exists a dimensional constant $\ep(n)  \geq (10n)^{-10n^3}$ such that the following holds. Let $(M^n,g)$ be any complete Riemannian manifold (not necessarily bounded curvature) with nonnegative Ricci curvature $\Rc_{g}\geq 0$. Suppose that $(M^n,g)$ has positive asymptotic volume ratio
\[
\lim_{r\to \infty} \frac{\Vol_g B_g(x_0, r)}{\omega_n r^n}\geq v>0,
\]
and curvature concentration satisfying
\be\label{linearrelationshipmainthm}
\int_M |\text{Rm}|_g^{\frac{n}{2}}\,dV_g\leq \ep(n) v.
\ee
Then $(M^n,g)$ is isometric to $(\R^n, g_E)$. 
\end{thm}

\vspace{5pt}

Even though Theorem \ref{main} is stated for dimensions $n\geq 4$, there is no generality lost with this assumption because it follows from work of Cheeger {\cite[Theorem 6.10]{Cheeger}} and Colding's Volume Continuity Theorem \cite{Colding} that for any odd $n$ (in particular, $n=3$), merely the assumption of $\|\Rm\|_{L^{\frac{n}{2}}(g)}<\infty$ is sufficient to guarantee that any nonnegative Ricci curvature manifold with Euclidean volume growth is necessarily isometric to $(\R^n,g_E)$. When $n$ is even, the picture is less clear. For instance, we have the example of the Eguchi-Hanson metric which is a nonflat 4-dimensional Ricci flat manifold with $\AVR\geq \frac{1}{2}$ and $\|\Rm\|_{L^{\frac{n}{2}}}<\infty$. For the dimensions $n$ in which such examples exist, there is a well-defined function $f_n:(0,1]\to \R$ which defines the ``size of the gap" for manifolds with at least the given asymptotic volume ratio, i.e., $f_n(v)$ is the supremum of all $\rho>0$ for which
\[
\left\{(M^n,g) \text{ with } \Rc_g\geq 0, \;\AVR(g)\geq v \text{ and } \int_M |\Rm|_g^{\frac{n}{2}}\,dV_g\leq \rho \right\}=\{(\R^n, g_E)\},
\]
where both of the sets of Riemannian manifolds above is understood in the sense of equivalent classes modulo isomorphism. The general behaviour of $f_n(v)$ is not currently well understood. For instance, it is clear that $v\mapsto f_n(v)$ is nondecreasing, but it is not known (to the author's knowledge) if we can even say that $\lim_{v\to 0^+}f_n(v)=0$. In light of Theorem \ref{main} however, we have the following natural question. 

\vspace{5pt}

\begin{question}\label{question1}
Is the linear relationship between the required smallness of the curvature concentration and the asymptotic volume ratio $v$ in Theorem \ref{main} optimal? In other words, is it true that 
\[
\lim_{v\to 0^+} \frac{f_n(v)}{v} <\infty?
\]
\end{question}

\vspace{5pt}

Some evidence towards an affirmative answer to this question is the relationship between between the curvature concentration and an $L^2$-Sobolev inequality via H\"older's inequality: 
\[
\int_M |\Rm| u^2\leq \left(\int_M |\Rm|^{\frac{n}{2}}\right)^{\frac{2}{n}}\left(\int_M u^{\frac{2n}{n-2}}\right)^{\frac{n-2}{n}}\lesssim v^{\frac{2}{n}}\left(\int_M u^{\frac{2n}{n-2}}\right)^{\frac{n-2}{n}}.
\]
This coefficient of $v^{\frac{2}{n}}$ on the RHS turns out to be the largest one possible that still allows the proof to run because our translations between volume ratios and Sobolev constants is sharp with respect to $v$ (see Section \ref{sectionsobolev}) and the particular power of $\frac{2}{n}$ is precisely the one that is necessary to eliminate the scalar curvature term in the $\nu$-functional as well as the $|\Rm|^{\frac{n}{2}+1}$ term in the evolution of $|\Rm|^{\frac{n}{2}}$ along the Ricci flow (see the proof of Theorem \ref{main2}). This is suggestive that the linear relationship of curvature concentration to $v$ as in \eqref{linearrelationshipmainthm} may be the largest one that the Ricci flow approach can handle.\\

The way in which we are able to obtain an explicit bound on $\ep(n)$ is to demonstrate the existence of a long-time solution to the Ricci flow with faster than $1/t$ curvature decay, which allows us to shift the limiting contradiction argument to time infinity. It is worth noting in the following result that we obtain \emph{no control on the rate} of which $t|\Rm|_{g(t)}\to 0$. This is the sacrifice we make in exchange for obtaining an explicit bound on the constant $\ep(n)$ in the statement of Theorem \ref{main}.

\vspace{5pt}

\begin{thm}\label{main2}
For any $n\geq 4$ and $v>0$, there exists $c(n)>0$, $\iota(n,v)>0$ such that the following holds. Suppose that $(M^n,g)$ satisfies the hypotheses of Theorem \ref{main} with this choice of $n$, $v$. Then there exists a long-time solution $g(t)$, $t\in [0,\infty)$ to the Ricci flow with $g(0)=g$ which satisfies:
\begin{enumerate}[(i)]
\item $\sR_{g(t)}\geq 0$ for every $t\in [0,\infty)$; \label{main2cond5}
\item $|\Rm|_{g(t)}\leq \frac{1}{t}$ for every $t\in (0,\infty)$; \label{main2cond1}
\item $\inj_{g(t)}\geq  \iota(n,v) \sqrt{t}$ for every $t\in (0,\infty)$;  \label{main2cond2}
\item $\Vol_{g(t)}B_{g(t)}(x, r)\geq c(n)v\,r^n$ for every $x\in M$, $r>0$,  $t\in (0,\infty)$;  and \label{main2cond3}
\item $\lim_{t\to \infty}  t |\Rm|_{g(t)}=0$. \label{main2cond4} 
\end{enumerate}
\end{thm}

\vspace{5pt}

The proof of Theorem \ref{main} follows directly from Theorem \ref{main2} and distance distortion along the Ricci flow, see section \ref{proofOfMain} for the details.

\vspace{5pt}

If we knew in advance that $(M^n, g)$ had \emph{bounded curvature}, we could circumvent the need for Theorem \ref{main2} in the proof of Theorem \ref{main} by directly applying the results from \cite{CM}. In particular, by Corollary \ref{SobfromAVRcor}, we could obtain the Sobolev inequality
\[
\left(\int_M |u|^{\frac{2n}{n-2}}\,dV_g\right)^{\frac{n-2}{n}}\leq C(n) v^{-\frac{2}{n}} \int_M |\nabla^g u|^2_g \,dV_g
\]
for all compactly supported $u\in W^{1,2}(M)$. If we restrict our $\ep(n)$ small enough so that 
\[
\ep(n)\leq  \left(\frac{\delta(n)}{C(n)}\right)^{\frac{n}{2}}
\]
where $\delta(n)$ is the dimensional constant in {\cite[Theorem 1.2]{CM}}, then by the conclusion of that theorem, we know that the complete bounded curvature Ricci flow $g(t)$ with initial condition $g(0)=g$, exists for all times $t\in [0,\infty)$ and satisfies the curvature estimates
\[
\sup_{x\in M} |\text{Rm}|_{g(t)}(x)\leq \frac{C_0(n)}{t},
\]
for any $t\in (0, \infty)$ and 
\[
\lim_{t\to \infty} t \sup_{x\in M} |\Rm|_{g(t)}(x) =0.
\]
This gives conditions akin to \eqref{main2cond1} and \eqref{main2cond4}. Moreover, by {\cite[Corollary 4.3]{CM}} and Lemma \ref{volfromSob}, we know that there exists a $v_0(n,v)>0$ such that 
\[
\Vol_{g(t)}B_{g(t)}(x, r)\geq v_0(n,v) r^n\; \text{ for any } \; t\geq 0, \;r>0, \; x\in M, 
\]
which gives \eqref{main2cond3}. Penultimately, for any $t>0$, the conditions
\[
\Vol_{g(t)} B_{g(t)}(x, \sqrt{t})\geq v_0 t^{\frac{n}{2}} \;\text{ and } \; |\Rm|_{g(t)}\leq \frac{C_0}{t},
\]
together imply the injectivity radius bound 
\[
\inj_{g(t)}(x) \geq  \iota(n,v_0) \sqrt{t}
\]
by \cite{CGT} for some $\iota(n,v_0)>0$ which does not depend on $t>0$. This gives condition \eqref{main2cond2}. Finally, condition \eqref{main2cond5} follows immediately from the assumption that $g(t)$ is a bounded curvature flow (so in particular, lower scalar curvature bounds are preserved by the maximum principle). The proof of Theorem \ref{main} then follows synonymously (see section \ref{proofOfMain}).\\

The issue that arises when $(M,g)$ has unbounded curvature is \emph{not} establishing short-time existence of the flow. Indeed, by {\cite[Theorem 4.1]{CHL}}, we can say that if $(M,g)$ satisfies the hypotheses of Theorem \ref{main} (for some $v>0$), then for any $\alpha>0$ there is a Ricci flow $g(t)$ with $g(0)=g$, defined on some short time interval $[0,T(\alpha, g)]$, which enjoys $\alpha/t$ curvature control. Rather, the issues that we primarily face in this case are: 1. Extending the flow for all times $t\in[0,\infty)$; and 2. Establishing the faster than $1/t$ curvature decay as in \eqref{main2cond4} in Theorem \ref{main2}. We get around both of these issues by showing that there exists a\footnote{The question of uniqueness of Ricci flows emanating from unbounded curvature complete metrics is still unanswered in general.} Ricci flow $g(t)$ which preserves the Sobolev inequality and curvature concentration of the initial metric (up to dimensional constants). Let us explain this in slightly more detail: Just as mentioned in the bounded curvature case above, we obtain from an asymptotic volume ratio at least $v>0$, the Sobolev inequality 
\[
\left(\int_{M} |u|^{\frac{2n}{n-2}}\,dV_g\right)^{\frac{n-2}{n}}\leq C(n) v^{-\frac{2}{n}}\int_{M} |\nabla^g u|^2_g \,dV_g.
\]
The Ricci flow that we construct will satisfy the properties
\be\label{flowpreservedpropertiesintro}
\begin{cases}
&|\Rm|_{g(t)}\leq \frac{1}{t}, \\
&\int_M |\Rm|_{g(t)}^{\frac{n}{2}}\,dV_{g(t)}\leq C_1(n) \ep(n) v, \text{ and } \\
&\left(\int_{M} |u|^{\frac{2n}{n-2}}\,dV_{g(t)}\right)^{\frac{n-2}{n}}\leq C_1(n)C(n) v^{-\frac{2}{n}}\int_{M} |\nabla^{g(t)} u|^2_{g(t)} \,dV_{g(t)}.
\end{cases}
\ee

This preservation of the Sobolev inequality and smallness of the curvature concentration allows us to apply {\cite[Theorem 1.2]{CM}} with the bounded curvature metric $g(1)$ replacing $g(0)$, from which we obtain the faster than $1/t$ curvature decay. \\

The organization of this paper is as follows. In section \ref{proofOfMain}, we prove that Theorem \ref{main} follows from Theorem \ref{main2}. In section \ref{sectionsobolev} we investigate the relationship between volume ratios, $L^2$-Sobolev inequalities, and control on the $\bar\nu$-functional in spaces with nonnegative Ricci curvature. In fact, we prove that these three notions of Sobolev control are equivalent in this setting. In section \ref{sectiondistance} we discuss how distances evolve under the Ricci flow, proving in particular, a version of the so-called ``Expanding Balls Lemma" where the time interval depends only on the dimension, $\alpha/t$ curvature bound, and lower bound on scalar curvature. In section \ref{sectionboundedcurv}, we prove a short-time Ricci flow existence result for \emph{bounded curvature} metrics with the property that the bound on the local curvature concentration is preserved (up to a dimensional constant). Finally in section \ref{sectionproofofmain}, we prove Theorem \ref{main2} by approximating blow-downs of $(M,g)$ by bounded curvature metrics, from which our short-time bounded curvature Ricci flow existence result is readily applied. This proves the existence of a long-time Ricci flow solution emanating from $g$ with the desired properties given in \eqref{flowpreservedpropertiesintro}, and thus also with the faster than $1/t$ curvature decay.

\vspace{5pt}

\section*{Acknowledgement}

I would like to thank Albert Chau for the helpful comments on earlier drafts of this paper, as well as Man-Chun Lee for the discussions of related problems at CRM.

\vspace{5pt}

\section{Proof of Theorem \ref{main} from Theorem \ref{main2}}\label{proofOfMain}

Here we provide a brief proof of our main result Theorem \ref{main} from our auxiliary long-time Ricci flow existence result, Theorem \ref{main2}. The method of proof will be familiar to experts; due to the scale-invariant properties of the hypotheses of Theorem \ref{main2}, we may apply Hamilton's Compactness Theorem to a sequence of blow-down flows to extract a converging subsequence on $(0,1]$. Then the faster than $1/t$ curvature decay and volume estimates together imply that this limiting flow must be isometric to the flat Euclidean flow on $\R^n$. The rigidity of the original metric can subsequently be concluded using distance distortion of Ricci flows. Here are the details.  

\begin{proof}[Proof of Theorem \ref{main} (assuming Theorem \ref{main2})]
Let $g(t)$, $t\in [0,\infty)$ be the long-time Ricci flow emanating from $g$ given by Theorem \ref{main2}. Now take any sequence of times $t_i\to \infty$, and define the sequence of parabolically rescaled flows $g_i(t)=\frac{1}{t_i}g(t_i t)$, each defined on $[0,1]$. Note that the curvature bound \eqref{main2cond1} and the injectivity radius bound \eqref{main2cond2} survive this rescaling, and so by Hamilton's Compactness Theorem \cite{Hamilton} and a diagonal argument, the $(M, g_i(t), x_0)$ subconverges (in the pointed Cheeger-Gromov sense) to some limiting complete flow $(\widetilde M, \tilde g(t), x_\infty)$, $t\in (0, 1]$ which, by nature of the faster than $1/t$ curvature decay \eqref{main2cond4}, must satisfy 
\[
|\Rm|_{\tilde g(t)}\equiv 0 \;\text{ for each } t\in (0,1]. 
\]
Furthermore, each $\tilde g(t)$ must have a positive asymptotic volume ratio because \eqref{main2cond3} passes to the limit as well. Therefore, for each $t\in (0,1]$, we must have that $(\widetilde M, \tilde g(t))$ is isometric to $(\R^n, g_E)$. Thus if we fix any $\delta>0$ and $t\in (0, 1]$ small enough (depending on $\delta$ and $n$), we can say that for all sufficiently large $i\in\N$, there holds 
\ba
\Vol_{g_i(0)}B_{g_i(0)}(x_0, 1)& \geq\Vol_{g_i(t)}B_{g_i(0)}(x_0, 1)\\&\geq 
\Vol_{g_i(t)}B_{g_i(t)}(x_0, 1-\beta \sqrt{t})\\&\geq
(1-\delta)\omega_n (1-\beta \sqrt{t})^n\\&\geq 
(1-2\delta)\omega_n
\end{align*}
where we have used the nonnegative scalar curvature of the flow \eqref{main2cond5} in the first inequality, the Shrinking Balls Lemma {\cite[Corollary 3.3]{ST}} in the second (here $\beta=\beta(n)$ is a dimensional constant), the convergence $g_i(t)\to g_E$ in the third, and our smallness condition on $t$ in the fourth. Scaling back tells us that 
\[
\lim_{r\to \infty}\frac{\Vol_g B_g(x_0, r)}{ r^n}\geq (1-2\delta) \omega_n.
\]
Since $\delta>0$ was arbitrary, we can use our assumption of $\Rc_{g}\geq 0$ again to in fact see that 
\[
\lim_{r\to \infty}\frac{\Vol_g B_g(x_0, r)}{r^n}= \omega_n.
\]
Therefore $(M^n,g)$ is isometric to $(\R^n, g_E)$ by the rigidity statement in the Bishop-Gromov Volume Comparison Theorem. 
\end{proof}

\vspace{5pt}

\section{Relationship between Sobolev quantities}\label{sectionsobolev}

\vspace{5pt}

In this section, we will investigate the relationship between three different notions of ``Sobolev control" in manifolds which enjoy lower Ricci bounds: volume ratios, $L^2$-Sobolev inequality, and control on the $\bar\nu$-functional. We will see that when $\Rc_g\geq 0$, all three of these notions of Sobolev control are equivalent (up to dimensional constants). 

\vspace{5pt}

For the sake of brevity, we introduce the following notation. For a connected $\Omega\subset (M,g)$, we will write
\[
\Sob(\Omega)\leq [A,B]
\]
to mean that the $L^2$-Sobolev inequality
\[
\left(\int_{\Omega} |u|^{\frac{2n}{n-2}}\,dV_g\right)^{\frac{n-2}{n}}\leq A\left( \int_{\Omega} |\nabla^g u|^2_g +B u^2\,dV_g\right)
\]
holds for all $u\in W^{1,2}(\Omega)$ with $\supp \,u\subset \subset \Omega$. We may also write $\Sob(\Omega,g)$ to emphasize which metric is being used.

\vspace{5pt}

The fact that an $L^2$-Sobolev inequality implies control on volume ratios is well-known, but since we care about the exact relationship between these quantities, we will provide a short proof. The following is only a slight modification of {\cite[Theorem 6.1]{Ye}}. 

\vspace{5pt}

\begin{lem}[Volume from Sobolev] \label{volfromSob}
Let $(M^n,g)$, $n\geq 3$ be a (possibly incomplete) Riemannian manifold with $B_g(x,L)\subset \subset M$ for some $L>0$. Assume that $\Sob(B_g(x,L))\leq [A, B]$ for some $A>0$, $B\geq 0$. Then whenever $r\leq L$, there holds
\[
\Vol_g B_g(x,r)\geq \left(\frac{1}{A(2^{n+3}+2BL^2)}\right)^{\frac{n}{2}}r^n.
\]
\end{lem}

\begin{proof}
Assume for the sake of contradiction that there exists $(M^n,g,x)$, $L>0$, $A>0$, $B\geq 0$ for which the hypotheses hold, and some $r\in (0,L]$ such that 
\[
\Vol_g B_g(x,r)< \left(\frac{1}{A(2^{n+3}+2BL^2)}\right)^{\frac{n}{2}}r^n=:\delta r^n.
\]
Set $\tilde g=r^{-2}g$ so that $\Vol_{\tilde g}B_{\tilde g}(x,1)<\delta$. Then by the assumed Sobolev inequality for the original metric $g$ and the fact that $r\leq L$, we obtain $\Sob(B_{\tilde g} (x,1))\leq [A, BL^2]$. Now by H\"older's inequality, $\Vol_{\tilde g} B_{\tilde g}(x,1)<\delta$, and the definition of $\delta$, we have
\[
ABL^2\int_{B_{\tilde g}(x,1)} u^2 \,dV_{\tilde g}\leq ABL^2\delta^{\frac{2}{n}}\left(\int_{B_{\tilde g}(x,1)} |u|^{\frac{2n}{n-2}} \,dV_{\tilde g}\right)^{\frac{n-2}{n}}\leq \frac{1}{2}\left(\int_{B_{\tilde g}(x,1)} |u|^{\frac{2n}{n-2}} \,dV_{\tilde g}\right)^{\frac{n-2}{n}}.
\]
Therefore this term can be absorbed into the LHS of the Sobolev inequality for $\tilde g$ to give $\Sob(B_{\tilde g} (x,1))\leq [2A, 0]$. The rest of the proof then is verbatim to the latter half of the proof of {\cite[Theorem 6.1]{Ye}}.
\end{proof}

\vspace{5pt}

Now we recall the definition of local entropy formulated by Wang \cite{WangA}: Let $(M^n,g)$ be a complete Riemannian manifold of dimension $n$, and $\Omega\subset M$ a connected open subset with smooth boundary and define
\ba
&\mathscr{S}(\Omega):=\left\{ u \bigg| u\in W^{1,2}_0(\Omega), \;\; u\geq 0, \;\; \int_\Omega u^2 \,dV_g=1\right\}, \\
&\overline{W}(\Omega, g,u, \tau):=\int_\Omega 4\tau |\nabla^g u|_g^2 -u^2 \log u^2 \,dV_g -\frac{n}{2}\log (4\pi \tau) -n, \\
&W(\Omega, g,u, \tau):=\int_\Omega \tau (4|\nabla^g u|_g^2+\sR_g u^2) -u^2 \log u^2 \,dV_g -\frac{n}{2}\log (4\pi \tau) -n, \\
&\bar{\mu}(\Omega, g, \tau):= \inf_{u\in \mathscr{S}(\Omega)} \overline{W}(\Omega, g,u, \tau), \;\;\; \mu(\Omega, g, \tau):= \inf_{u\in \mathscr{S}(\Omega)} W(\Omega, g,u, \tau), \\
&\bar{\nu}(\Omega, g, \tau):= \inf_{s\in (0, \tau]}\bar{\mu}(\Omega, g, s), \;\;\; \nu(\Omega, g, \tau):= \inf_{s\in (0, \tau]}\mu(\Omega, g, s), \\
&\bar{\nu}(\Omega, g):= \inf_{s\in (0, \infty)}\bar{\mu}(\Omega, g, s), \;\;\; \nu(\Omega, g):= \inf_{s\in (0, \infty)}\mu(\Omega, g, s).
\end{align*}

In the above (and throughout the paper), $\sR_g$ denotes the scalar curvature of $g$. Although these definitions are given for domains $\Omega$ with \emph{smooth} boundaries, we will often take the domain to be some geodesic ball within $M$, which will not necessarily have smooth boundary. This slight abuse of notation is validated for our purposes by {\cite[Lemma 2.6 and Remark 2.7]{WangA}}. \\

The following shows the implication of $\bar\nu$-functional control given a Sobolev inequality. Its proof is a simple consequence of Jensen's inequality (see for instance \cite{Ye}, \cite{ChenEric}, \cite{CM}). 

\vspace{5pt}

\begin{lem}[$\bar \nu$-functional from Sobolev]\label{nufromSob}
For any $n\geq 3$, there exists a dimensional constant $C(n)>0$ such that the following holds. Let $(M^n,g)$ be a (possibly incomplete) Riemannian manifold with $\Omega\subset\subset M$ and $\Sob(\Omega)\leq [A,\frac{a}{\tau}]$ for some $A\geq 1$, $a\geq 0$, $\tau>0$. Then $\bar\nu(\Omega, g,\tau)\geq -\frac{n}{2}\log(C(n) A)-4a$. If in addition $\int_{\Omega}|\Rm|_g^{\frac{n}{2}}\,dV_g\leq \left(\frac{2}{n A}\right)^{\frac{n}{2}}$, or more generally $\int_{\Omega}\sR_g^{\frac{n}{2}}\,dV_g\leq \left(\frac{2}{ A}\right)^{\frac{n}{2}}$, then $ \nu(\Omega, g, \tau)\geq -\frac{n}{2}\log(C(n) A)-4a$.
\end{lem}

\begin{proof}
Just as in the proof of {\cite[Lemma 3.1]{CM}}, Jensen's inequality and the fact that $\log x\leq  bx-1-\log b$ (for any $b>0$) implies
\be\label{2.2proof}
\begin{split}
\int_\Omega u^2 \log u^2\,dV_g&\leq\frac{n}{2}\log\left(\int_\Omega  |\nabla^g u|^2_g+\frac{a}{\tau}u^2\,dV_g\right)+\frac{n}{2}\log A\\&\leq 
4s\left(\int_\Omega  |\nabla^g u|^2_g+\frac{a}{\tau}u^2\,dV_g\right)+\frac{n}{2}\log A-1-\log \left(\frac{8s}{n}\right).
\end{split}
\ee
This shows that $\bar\nu(\Omega, g,\tau)\geq -\frac{n}{2}\log(C_1(n) A)-4a$ since $\frac{4sa}{\tau}\leq 4a$ when $s\leq \tau$. For the other implication, we apply H\"older's and our Sobolev inequality
\[
2s\int_\Omega \sR_g u^2\,dV_g\geq -4s\left(\int_\Omega|\nabla^g u|^2_g+\frac{a}{\tau}u^2\,dV_g\right).
\]
Putting this back into \eqref{2.2proof} gives 
\[
\int_\Omega u^2 \log u^2\,dV_g\leq 8s\left(\int_\Omega  |\nabla^g u|^2_g+\frac{\sR_g}{4}u^2+\frac{a}{\tau}u^2\,dV_g\right)+\frac{n}{2}\log A-1-\log \left(\frac{8s}{n}\right).
\]
Now relabel $\tilde s=2s$ and restrict $\tilde s\leq \tau$ to obtain $\nu(\Omega, g, \tau)\geq -\frac{n}{2}\log(2C_1(n) A)-4a$. For simplicity, we will assume $C(n)=2C_1(n)$ from which both implications clearly follow. 
\end{proof}

\vspace{5pt}

It is also well-known that rough control on the $\bar\nu$-functional on a set $\Omega$ is sufficient to obtain a Sobolev inequality on $\Omega$. We include a precise statement for our ease of reference. 

\vspace{5pt}

\begin{lem}[Sobolev from $\bar{\nu}$-functional]\label{Sobfromnu}
For any $n\geq 3$, there exists a dimensional constant $C(n)>0$ such that the following holds. Suppose $(M^n,g)$, is a (possibly incomplete) Riemannian manifold, $\Omega\subset\subset M$ is connected, and $A,\tau>0$.
\begin{enumerate}
\item If $\bar{\nu}(\Omega, g, \tau)\geq -\frac{n}{2}\log A$, then $\Sob(\Omega)\leq [C(n)A,\frac{1}{\tau}]$.\label{sobfromnupart1}
\item If $\nu(\Omega, g, \tau)\geq -\frac{n}{2}\log A$ and $\int_{\Omega}|\Rm|_g^{\frac{n}{2}}\,dV_g\leq \left(\frac{2}{nC(n)A}\right)^{\frac{n}{2}}$, or more generally $\int_{\Omega}\sR_g^{\frac{n}{2}}\,dV_g\leq \left(\frac{1}{ C(n)A}\right)^{\frac{n}{2}}$, then $\Sob(\Omega)\leq [2C(n)A,\frac{1}{\tau}]$.\label{sobfromnupart2}
\end{enumerate}
\end{lem}

\begin{proof}
We will only discuss the proof of part \eqref{sobfromnupart2} since part \eqref{sobfromnupart1} is simpler. Fix some $\lambda>0$, a compact $K\subset \subset \Omega$, and observe that the condition on the $\nu$-functional implies
\ba
\int_\Omega u^2 \log u^2 \,dV_g&\leq s \left(\int_\Omega |\nabla^g u|_g^2+\frac{\sR_g }{4} u^2 \,dV_g\right) -\frac{n}{2}\log s+\frac{n}{2}\log A-\frac{n}{2}\log (\pi) -n\\&\leq  s \left(\int_\Omega |\nabla^g u|_g^2+\frac{\sR_g^+ }{4} u^2+\lambda u^2 \,dV_g\right) -\frac{n}{2}\log s+\frac{n}{2}\log A-\frac{n}{2}\log (\pi) -n
\end{align*}
for all $s\in (0,4\tau]$ and $u\in W^{1,2}_0(K)$ with $\int_K u^2\,dV_g=1$. Now by {\cite[Theorem 5.5]{Ye}}, we may conclude a Sobolev inequality 
\be\label{lem3e1sob}
\left(\int_K u^{\frac{2n}{n-2}}\,dV_g\right)^{\frac{n-2}{n}}\leq C(n,A) \left(\int_K |\nabla^g u|^2_g +\frac{\sR_g^+}{4} u^2+\frac{1}{\tau}u^2+\lambda u^2\,dV_g\right)
\ee
for all $u\in W^{1,2}_0(K)$. Even though the results in \cite{Ye} are stated for closed manifolds, equivalent statements hold for compact manifolds with boundary (such as our $K\subset\subset \Omega$). Now since $K$ is compact, we may send $\lambda\to 0$. The fact that the constant $C(n,A)$ is of the precise form $C(n)A$ follows by tracking through the constants in {\cite[Theorems 5.3-5.5]{Ye}} as was first noticed by Chen \cite{ChenEric} (see also {\cite[Appendix]{CM}}). So it remains to deal with the remaining scalar curvature term. For this, we apply H\"older's inequality
\[
\frac{C(n) A}{4}\int_K \sR_g^+ u^2\,dV_g\leq \frac{C(n) A}{4}\left(\int_K |\sR_g|^{\frac{n}{2}}\,dV_g\right)^{\frac{2}{n}}\left(\int_K u^{\frac{2n}{n-2}}\,dV_g\right)^{\frac{n-2}{n}}\leq \frac{1}{2}\left(\int_K u^{\frac{2n}{n-2}}\,dV_g\right)^{\frac{n-2}{n}}.
\]
So we can absorb this term into the LHS of \eqref{lem3e1sob} to give $\Sob(K)\leq [2C(n)A, \frac{1}{\tau}]$. The result follows by taking an exhaustion of compact $K_1\subset K_2\subset \dots \subset \Omega$. 
\end{proof}

\vspace{5pt}

\begin{remark}\label{productofconstantsbound}
In our application, we will have some Ricci flow $g(t)$, $t\in [0,T]$, an initial Sobolov inequality $\Sob(\Omega,g(0))\leq [A,0]$, and some assumed smallness of $\int |\Rm|_{g(0)}^{\frac{n}{2}}\,dV_{g(0)}$. Lemma \ref{nufromSob} then will give us $\nu$-functional control $\nu(\Omega, g(0), \tau+T)\geq -\frac{n}{2}\log(C_1(n) A)$, which - in our application where our flows enjoy $c/t$ curvature control - is preserved under the Ricci flow to give $\nu(\Omega, g(t), \tau)\geq -\frac{n}{2}\log(C_1(n) A)$. Then Lemma \ref{Sobfromnu} will return us to a Sobolev inequality for $g(t)$: $\Sob(\Omega, g(t))\leq [2C_2(n) C_1(n) A, \frac{1}{\tau}]$. It follows from {\cite[Proof of Theorem 3.4 and Appendix]{CM}} that $C_1(n)C_2(n)\leq 250n$. We will use this estimate in obtaining a bound for our constant $\ep(n)$.
\end{remark}

\vspace{5pt}

In the presence of a (possibly negative) lower Ricci curvature bound $\Rc_g\geq -K$, Wang {\cite[Theorem 3.6]{WangA}} proved that one can obtain local control on the $\bar\nu$-functional depending on the volume ratio:
\be\label{nufromlowerriccibound}
\bar\nu(B_g(x,r), g, r^2)\geq -10 n^2(1+K r)+\log\left(\frac{\Vol_g B_g(x,r)}{\omega_n r^n}\right)^{n+1}.
\ee
Here $\omega_n$ is the volume of an $n$-dimensional Euclidean unit ball. In particular, when $\Rc_g\geq 0$, we can obtain global control on the $\bar\nu$-functional (and subsequently a global Sobolev inequality) if we know that \emph{all the volume ratios} of $(M,g)$ are uniformly bounded from below. Such a manifold is said to have \emph{positive asymptotic volume ratio}, which we will denote by
\[
\AVR(g):=\lim_{r\to \infty}\frac{\Vol_g B_g(x_0,r)}{\omega^n r^n}.
\]
It is well-known that when $\Rc_g\geq 0$, $\AVR(g)$ does not depend on the choice of fixed point $x_0$, is a well-defined number in $[0,1]$, and that $\AVR(g)=1$ if and only if $(M^n,g)$ is isometric to $(\R^n, g_E)$ where $g_E$ is the Euclidean metric. \\

Now suppose that $(M,g)$ were some nonnegative Ricci curvature manifold with a known volume bound
\[
\Vol_g B_g(x,1)\geq v>0.
\]
Then by \eqref{nufromlowerriccibound} we may obtain $\bar\nu(B_g(x,1),g,1)\geq -C+\log (v^{\frac{n+1}{n}})$, which in turn gives $\Sob(B_g(x,1))\leq [C'v^{-\frac{2(n+1)}{n}}, 1]$ by Lemma \ref{Sobfromnu}. Now if we return to a bound on the volume via Lemma \ref{volfromSob}, we obtain
\[
\Vol_g B_g(x,1)\geq c v^{n+1},
\]
which is a much worse volume estimate than we started with as $v\to 0$. Therefore one step along this process is not sharp with respect to $v$. It turns out that this bluntness occurs in going from volume to $\bar\nu$-functional in \eqref{nufromlowerriccibound}, and is rectified in the following lemma. Put another way, Lemma \ref{SobfromAVR2} completes the circle of equivalences between volume ratios, Sobolev inequality, and control on the $\bar\nu$-functional (up to dimensional constants). 

\vspace{5pt}

\begin{lem}[Sobolev from volume]\label{SobfromAVR2}
For any $n\geq 3$, there exists a dimensional constant $C(n)>0$ such that the following holds. Let $(M^n,g)$ be a complete Riemannian manifold (not assumed bounded curvature) such that $\Rc_{g}\geq 0$ and $\AVR(g)>0$. If for some $r>0$ and $v>0$ there holds
\be\label{SobfromAVRvolumeratio}
\frac{\Vol_g B_g(x,r)}{r^n}= v, 
\ee
then $\Sob(B_g(x,r))\leq [C(n)v^{-\frac{2}{n}}, 0]$.
\end{lem}

\vspace{5pt}

\begin{remark}
The constant $v$ in the volume ratio hypothesis \eqref{SobfromAVRvolumeratio} is permitted to be much larger than the value of $\omega_n\cdot \AVR(g)$. In fact, the assumption that $\AVR(g)>0$ in Lemma \ref{SobfromAVR2} is likely superfluous because the proof of {\cite[Theorem 1.1]{Brendle}} (which we will be modifying) is local in nature until the end when they let $r\to \infty$ (instead we will take a determined value of $r$). Regardless, the stated version will be sufficient for our purposes. 
\end{remark}

\vspace{5pt}

\begin{proof}[Proof of Lemma \ref{SobfromAVR2}]
First note that by volume comparison, we have
\[
\frac{v}{2^n}=\frac{\Vol_g B_g(x,r)}{(2r)^n}\leq\frac{\Vol_g B_g(x,2r)}{(2r)^n}\leq \frac{\Vol_g B_g(x,r)}{r^n}= v.
\]
Therefore, it suffices to prove the lemma with $v$ replaced by 
\[
v_0:=\frac{\Vol_g B_g(x,2r)}{(2r)^n},
\]
at the expense of replacing the dimensional constant $C(n)$ with (at worst) $2^n C(n)$. \\

Fix some smooth nonnegative function $u\in C^\infty(B_g(x, r))$ with $u\equiv 0$ on $\partial B_g(x, r)$. Since the hypotheses and conclusion of the lemma are invariant under scaling, we may assume that
\be\label{Soblemscalecond}
\int_{B_g(x,r)}|\nabla^g u|_g\,dV_g=n \int_{B_g(x,r)} u^{\frac{n}{n-1}}\,dV_g.
\ee
If this seems like a somewhat mysterious condition, we refer the reader to the proof of {\cite[Theorem 1.1]{Brendle}}, which we will be modifying (note that the coefficient of $n$ in \eqref{Soblemscalecond} is only there to match the convention in \cite{Brendle}). We start with a couple elementary claims. 

\vspace{5pt}

\begin{claim}\label{Soblemclaim1}
$\Vol_g B_g(x,r)\leq (2r)^n (1-(\frac{1}{7})^n ) v_0$
\end{claim}

\begin{proof}[Proof of Claim]
Let $\gamma: [0,\infty)\to M$ be a geodesic ray with $\gamma(0)=x$. Setting $y=\gamma(3r/2)$, we therefore have
\[
B_g(y,r/2)\subset B_g(x,2r)\setminus B_g(x,r).
\]
Then by volume comparison and the fact that $B_g(y,\frac{7}{2}r)\supset B_g(x,2r)$, we can see that
\[
\frac{\Vol_g B_g(y,r/2)}{(r/2)^n}\geq \frac{\Vol_g B_g(y,7r/2)}{(7r/2)^n}\geq \frac{\Vol_g B_g(x,2r)}{(7r/2)^n}\geq \frac{\Vol_g B_g(x,2r)}{(2r)^n}(4/7)^n= v_0 (4/7)^n.
\]
Then we obtain the upper bound 
\[
\Vol_g B_g(x,r)\leq \Vol_g B_g(x,2r)- \Vol_g B_g(y,r/2)\leq (2r)^n v_0 -(2r/7)^n v_0=(2r)^n (1-(1/7)^n) v_0.
\]
\end{proof}

\begin{claim}\label{Soblemclaim2}
There exists dimensional constants $\delta(n)>0$, $C_1(n)\geq 1$ such that for any measurable nonnegative function $f: B_g(x,r)\to \R$, there holds 
\[
\frac{1}{(2r)^n}\int_{B_g(x,r)}(1+f(x))^n\,dV_g \leq (1-\delta) v_0 +\frac{C_1}{(2r)^n} \int_{B_g(x,r)} f(x)^n\,dV_g.
\]
\end{claim}

\begin{proof}[Proof of Claim]
Applying the Binomial Theorem gives
\[
\int_{B_g(x,r)}(1+f(x))^n\,dV_g \leq \Vol_g B_g(x,r)+\sum_{k=1}^n {n \choose k}\int_{B_g(x,r)}f(x)^k\,dV_g.
\]
For the $1\leq k\leq n-1$ terms, we write $p_k=\frac{n}{k}$, $q_k=\frac{n}{n-k}$ and apply H\"older's then Young's inequalities 
\ba
\int_{B_g(x,r)}f(x)^k\,dV_g&\leq \left(\int_{B_g(x,r)}f(x)^n\,dV_g\right)^{\frac{1}{p_k}}(\Vol_g B_g(x,r) )^{\frac{1}{q_k}}\\&\leq 
\frac{\ep_k^{q_k}}{q_k}\Vol_g B_g(x,r)+\frac{1}{\ep_k^{p_k} p_k}\int_{B_g(x,r)}f(x)^n\,dV_g.
\end{align*}
We may choose the $\ep_k$ small enough so that 
\[
(1-(1/7)^n)\left(1+\sum_{k=1}^{n-1} {n\choose k}\frac{\ep_k^{q_k}}{q_k}\right)\leq 1-\frac{1}{2}(1/7)^n.
\]
Therefore, by applying Claim \ref{Soblemclaim1}, we can see that with $\delta(n)=\frac{1}{2}(1/7)^n$ and 
\[
C_1(n)=1+\sum_{k=1}^{n-1} {n\choose k}\frac{1}{\ep_k^{p_k} p_k},
\]
the claim holds. 
\end{proof}

\vspace{5pt}
 
Now we recall some of the details from \cite{Brendle}. The following is taken out of the proof of their main theorem (specifically right up until the end of their Corollary 2.6 and the subsequent few lines), which we state as an independent result for clarity. 

\vspace{5pt}

\begin{prop}[Proof of {\cite[Theorem 1.1]{Brendle}}]\label{Brendlemainthm}
Suppose that $(M,g)$ satisfies $\Rc_g\geq 0$ and $\AVR(g)>0$. Let $D\subset M$ be some compact connected region with smooth boundary and $u\in C^\infty(D)$ is nonnegative with $u\equiv 0$ on $\partial D$ and 
\[
\int_D |\nabla^g u|\,dV_g=n\int_D u^{\frac{n}{n-1}}\,dV_g.
\]
Then for all $\rho>0$, there holds
\[
\Vol_g \{y\in M : d_g(x,y)< \rho \;\text{ for all } x\in D\}\leq \int_D \left(1+\rho \,u^{\frac{1}{n-1}}\right)^n \,dV_g.
\]
\end{prop}

\vspace{5pt}

We apply Proposition \ref{Brendlemainthm} with\footnote{The smooth boundary condition can be accomplished by padding $B_g(x,r)$ slightly by taking $B_g(x,r)\subset D_\ep$, where $D_\ep$ has smooth boundary and is contained in the $\ep$-ball of $B_g(x,r)$ (see {\cite[Lemma 2.6]{WangA}}), then take $\ep\to 0$.}  $D=B_g(x,r)$ and $\rho=3r$ to obtain 
\[
\Vol_g B_g(x,2r)=\Vol_g \{y\in M : d_g(p,y)<3r \;\text{ for all } p\in B_g(x,r)\} \leq \int_{B_g(x,r)}\left(1+3r \,u^{\frac{1}{n-1}}\right)^n \,dV_g.
\]
Now divide both sides by $(2r)^n$ and apply Claim \ref{Soblemclaim2} with $f=3r u^{\frac{1}{n-1}}$ to obtain
\[
v_0\leq (1-\delta(n)) v_0 +\frac{C_1}{(2r)^n} \int_{B_g(x,r)} \left(3r\, u^{\frac{1}{n-1}}\right)^n\,dV_g.
\]
Rearranging and writing $C_2(n)=\frac{C_1 3^n}{\delta 2^n}$ gives
\[
v_0\leq C_2 \int_{B_g(x,r)} u^{\frac{n}{n-1}}\,dV_g.
\]
Therefore, by applying our scaling condition \eqref{Soblemscalecond}, we immediately see
\be\label{Sobineqe2}
\int_{B_g(x,r)}|\nabla^g u|_g\,dV_g=n \int_{B_g(x,r)} u^{\frac{n}{n-1}}\,dV_g\geq n v_0^{\frac{1}{n}}C_2^{-\frac{1}{n}}\left(\int_{B_g(x,r)} u^{\frac{n}{n-1}}\,dV_g\right)^{\frac{n-1}{n}}.
\ee
Thus we obtain an $L^1$-Sobolev inequality for the given function $u$, which is scale-invariant. So we may scale back to our original value of $r$ and therefore may assume that \eqref{Sobineqe2} holds for all nonnegative $u\in C^\infty(B_g(x,r))$ with $u\equiv 0$ on $\partial B_g(x,r)$. Now it is well-known that an $L^1$-Sobolev inequality implies an $L^2$-Sobolev inequality (cf. \cite{Li}). We provide a brief proof for the reader's convenience and to verify the proper dependence on $v_0$. By the $L^1$-Sobolev inequality \eqref{Sobineqe2} and Cauchy-Schwarz, we have
\ba
\left(\int_{B_g(x,r)} u^{\frac{2n}{n-2}} \,dV_g\right)^{\frac{2(n-1)}{n}}&=\left(\int_{B_g(x,r)} \left(u^{\frac{2(n-1)}{n-2}}\right)^{\frac{n}{n-1}} \,dV_g\right)^{\frac{2(n-1)}{n}}\\&\leq
\left(n v_0^{\frac{1}{n}}C_2^{-\frac{1}{n}}\right)^{-2} \left(\int_{B_g(x,r)} |\nabla^g u^{\frac{2(n-1)}{n-2}}|_g\,dV_g\right)^2\\&=
\left(\left(n  v_0^{\frac{1}{n}}C_2^{-\frac{1}{n}}\right)\inv \frac{2(n-1)}{n-2}\right)^2\left(\int_{B_g(x,r)} u^{\frac{n}{n-2}}|\nabla^g u|_g\,dV_g\right)^2\\&\leq
C(n) v_0^{-\frac{2}{n}}\left(\int_{B_g(x,r)} u^{\frac{2n}{n-2}} \,dV_g\right)\left(\int_{B_g(x,r)} |\nabla^g u|_g^2 \,dV_g\right).
\end{align*}
We obtain the desired $L^2$-Sobolev inequality by rearranging the above and approximation. This completes the proof of the lemma. 
\end{proof} 	

\vspace{5pt}

\begin{remark}\label{boundonSobfromvolconstant}
By taking the $\ep_k=(2\cdot 7^n n^{n+1})\inv$ in the proof of Claim \ref{Soblemclaim2}, we can calculate a naive bound on the dimensional constant in Lemma \ref{SobfromAVR2} when $n\geq 4$ of 
\[
C(n)\leq 2^n \underbrace{\left(\frac{2(n-1)}{n(n-2)}\right)^2}_{\leq 2} C_2^{\frac{2}{n}}\leq 2^{n+1}\left[\left(\frac{3}{2}\right)^n\underbrace{2\cdot 7^n}_{=1/\delta}\underbrace{n^{n+1}(2\cdot 7^n n^{n+1})^n}_{\leq C_1}\right]^{\frac{2}{n}}\leq (10n)^{3n}.
\]
\end{remark}

\vspace{5pt}

As previously mentioned, when $\Rc_g\geq 0$, a positive lower bound on the asymptotic volume ratio $\AVR(g)$ implies a lower bound on \emph{all} volume ratios by volume comparison. Therefore Lemma \ref{SobfromAVR2} immediately yields the following corollary. 

\vspace{5pt}

\begin{cor}\label{SobfromAVRcor}
For any $n\geq 3$ and $C(n)$ as in Lemma \ref{SobfromAVR2}, the following holds. Let $(M^n,g)$, $n\geq 3$ be a complete Riemannian manifold (not assumed bounded curvature) such that $\Rc_{g}\geq 0$ and $\AVR(g)>0$. Then $\Sob(M, g)\leq [C(n)(\omega_n \AVR(g))^{-\frac{2}{n}}, 0]$.
\end{cor}

\vspace{5pt}

\section{Evolution of distances and curvature concentration along the Ricci flow}\label{sectiondistance}

\vspace{5pt}

In this section, we first give two results which describe how the distance function behaves under the Ricci flow. In particular, these results will allow us to conclude that the radius of geodesic balls stay uniformly controlled both forward and backward in time. The first is the well-known ``Shrinking Balls Lemma" {\cite[Corollary 3.3]{ST}} which controlls the rate at which distances can contract.

\vspace{5pt}

\begin{lem}[Shrinking Balls \cite{ST}]\label{SBL}
For any $n\in\N$ and $\beta(n)=8\sqrt{2/3}(n-1)\geq 1$, the following holds. Suppose $(M^n,g(t))$, $t\in [0,T]$ is a (not necessarily complete) $n$-dimensional Ricci flow such that $B_{g(0)}(x_0, r)\subset \subset M$ and $|\Rm|_{g(t)}\leq c_0/t$ on $B_{g(0)}(x_0,r)\cap B_{g(t)}(x_0,r-\beta\sqrt{c_0 t})$ for each $t\in (0,T]$ and some $c_0>0$. Then for any $t\in [0,T]$, there holds 
\[
B_{g(t)}(x_0, r-\beta\sqrt{c_0 t})\subset B_{g(0)}(x_0, r).
\]
\end{lem}

\vspace{5pt}

Obtaining control in the other direction (i.e., the rate at which distances can expand) is historically more challening; If it is known that the flow enjoys a lower Ricci bound \emph{for all times}, then Simon \& Topping have proved an analogous ``Expanding Balls Lemma" {\cite[Lemma 3.1]{ST}}. The issue is however that lower Ricci curvature bounds are not respected under Ricci flow when $n\geq 4$. Therefore it is desirable to replace the lower Ricci bound assumption with a lower \emph{scalar curvature} bound (which is preserved) along with some volume estimates (see {\cite[Lemma 3.7]{He} and {\cite[Lemma 2.2]{LeeTam}} for such results of this flavour). The improvement that we offer here is uniform control on the rate of the expansion depending only on the quotient of volume ratios, albeit for a short amount of time. Importantly, the time interval does \emph{not} depend on the given volume bounds. 

\vspace{5pt}

\begin{lem}[Expanding Balls]\label{EBL}
For any $n \in \mathbb{N}$ and any $\alpha, \sigma > 0$, there exists $\widehat{T}(n, \alpha, \sigma) > 0$ such that the following holds. Let $(M^n, g(t))$, $t\in [0,T]$ be a complete bounded curvature Ricci flow such that for some $x_0\in M$, $b_0>0$, and $c_0>0$ there holds: 
\begin{enumerate}[(a)]
    \item $\inf_{x\in M}\sR_{g(0)}(x) \geq -\sigma$,
    \item $\sup_{x\in M}|\Rm|_{g(t)}(x) \leq \alpha t^{-1}$ for all $t\in (0,T]$,
    \item $\Vol_{g(0)}B_{g(0)}(x_0, 2) \leq b_0 $, and 
    \item $\Vol_{g(t)}B_{g(t)}(x, r) \geq c_0  r^n$ whenever $x\in B_{g(0)}(x_0, 1)$, $r\in (0,\frac{1}{2}]$, and $t\in (0,T]$.
\end{enumerate}
Then for all $t \in [0, T\land \widehat{T}]$, we have $B_{g(0)}(x_0, 1)  \subset B_{g(t)}(x_0, 2^{n+2} b_0c_0\inv)$.
\end{lem}

\begin{proof}
We begin by setting our maximal time interval:
\[
\widehat T= \min\left\{ \frac{\log 2}{\sigma}, \frac{1}{4\beta^2 \alpha}\right\}.
\]
Here $\beta=\beta(n)$ is from Lemma \ref{SBL}. Fix $\tau\in [0,T\land \widehat T]$ and define 
\[
R = \inf\{r>0 : B_{g(0)}(x_0, 1) \subset B_{g(\tau)}(x_0, r)\}.
\]
We will show that $R\leq 2^{n+2}b_0c_0\inv$ which will prove the result. By definition, there exists $y \in M$ such that $d_{g(0)}(x_0, y) = 1$ and $d_{g(\tau)}(x_0, y) = R$. Let $\gamma : [0, 1]\to M$ be a minimizing unit-speed $g(0)$-geodesic with $\gamma(0)=x_0$ and $\gamma(1)=y$. Notice that because $\widehat T\leq \frac{1}{4\beta^2 \alpha}$, we have that $B_{g(\tau)}(x, 1/2)\subset B_{g(0)}(x,1)$ for all $x\in M$ by Lemma \ref{SBL}. Thus by taking $z_1=x_0$, we can inductively choose some set of points $\{z_j\}_{j=1}^{k}\subset B_{g(0)}(x_0,1)$ such that $\{B_{g(\tau)}(z_j, \frac{1}{2})\}_{j=1}^k$ are mutually disjoint and
\[
\bigsqcup_{j=1}^{k} B_{g(\tau)}(z_j, 1/2) \subset B_{g(0)}(x_0,2)\;\; \text{ and } \;\; B_{g(0)}(x_0, 1) \subset \bigcup_{j=1}^{k} B_{g(\tau)}(z_j, 1).
\]

Then since $\gamma \subset \overline{B_{g(0)}(x_0, 1)}$, we can obtain a lower bound on $d_{g(\tau)}(x_0,y)$ by estimating the length of $\gamma$ with respect to the metric $g(\tau)$ as 
\be\label{EBLRboundedk}
R = d_{g(\tau)}(x_0, y) \leq \text{Length}_{g(\tau)}(\gamma)\leq 2k.
\ee

By our assumptions on the bounds of volumes and the assumed disjointness of $\{B_{g(\tau)}(z_j, \frac{1}{2})\}_{j=1}^k$, we can obtain an upper bound on $k$:
\be\label{EBLboundonk}
\begin{split}
k c_0  (1/2)^n &\leq \sum_{j=1}^k \Vol_{g(\tau)} B_{g(\tau)}(z_j, 1/2) \\&
\leq \Vol_{g(\tau)}(B_{g(0)}(x_0, 2)) \\&
\leq 2 \Vol_{g(0)}(B_{g(0)}(x_0, 2)) \leq 2 b_0.
\end{split}
\ee
In the third inequality, we have used the evolution of the volume form $\partial_t dV_{g(t)} = -\sR_{g(t)} dV_{g(t)}$, the lower scalar bound $\sR_{g(0)}\geq -\sigma$ (which is preserved along complete bounded curvature flows), and the condition on the existence time $\sigma \widehat T\leq \log 2$. By \eqref{EBLRboundedk} and \eqref{EBLboundonk}, we have the desired bound $R\leq 2^{n+2} b_0 c_0\inv$.
\end{proof}

\vspace{5pt}

\begin{remark}
Because local lower bounds on the scalar curvature $\sR_{g(0)}$ are semi-preserved along all Ricci flows with $\alpha/t$ curvature control \cite{BingLongChen}, it is clear that Lemma \ref{EBL} can be localized (i.e., for some compact region contained within a possibly incomplete and/or unbounded curvature flow), but the stated version will suffice for our purposes. 
\end{remark}

\vspace{5pt}

We end this section by giving a lemma demonstrating how the local curvature concentration evolves under the Ricci flow. The proof is essentially that of {\cite[Lemma 2.1]{CCL}}, but we provide a brief sketch to demonstrate explicit estimates on the coefficients.

\vspace{5pt}

\begin{lem}\label{curvconcevolutionlem}
Suppose $(M,g(t))$, $n\geq 3$, $t\in [0,T]$ is a complete (possibly unbounded curvature) Ricci flow and $\phi: M\times [0,T]\to \R$ is a smooth nonnegative function with compact support. If $\square \phi\leq 0$ (here $\square=\partial_t-\Delta^{g(t)}$ is the heat operator), then for any $\beta>0$ there holds
\begin{align*}
\frac{d}{dt}\int \phi^2&(|\Rm|_{g(t)}^2+\beta)^{\frac{n}{4}}\,dV_{g(t)}\leq 
-\frac{1}{3}\int  |\nabla(\phi (|\Rm|_{g(t)}^2+\beta)^{\frac{n}{8}})|^2 \,dV_{g(t)}\\&\;\;\;+
10n \int |\nabla^{g(t)} \phi|_{g(t)}^2 (|\Rm|_{g(t)}^2+\beta)^{\frac{n}{4}}\,dV_{g(t)}+5n\int \phi^2(|\Rm|_{g(t)}^2+\beta)^{\frac{n}{4}+\frac{1}{2}}\,dV_{g(t)}.
\end{align*}
\end{lem}

\begin{proof}
In this proof, we will denote $\alpha=\frac{n}{4}$ and we will not write the metric $g(t)$ in any terms except for the volume form $dV_t$. We begin by differentiating under the integral, integrating by parts (since $\supp\,\phi$ is compact), using the evolution of volume $\partial_t \,dV_t=-\sR\,dV_t$, and $|\sR|\leq n|\Rm|$ (see for instance {\cite[Proposition 7.28]{Lee}}) which gives
\begin{align*}
\frac{d}{ dt}\int \phi^2(|\text{Rm}|^2+\beta)^{\alpha}\,dV_t&\leq 
%\int \frac{\partial}{\partial t}\left(\phi^2(|\text{Rm}|^2+\beta)^{\alpha}\right)\,dV_t+\int \phi^2|\sR|(|\text{Rm}|^2+\beta)^{\alpha+\frac{1}{2}}\,dV_t\\&\leq
\int \square\left(\phi^2(|\text{Rm}|^2+\beta)^{\alpha}\right)\,dV_t+n\int \phi^2(|\text{Rm}|^2+\beta)^{\alpha+\frac{1}{2}}\,dV_t.
\end{align*}
For the $\square$ term, we calculate

\begin{align*}
\int \square&\left(\phi^2(|\text{Rm}|^2+\beta)^{\alpha}\right)\,dV_t=  
\int 2\phi \square\phi (|\text{Rm}|^2+\beta)^{\alpha}\,dV_t-\int 2|\nabla \phi|^2 (|\text{Rm}|^2+\beta)^{\alpha}\,dV_t\\&
+\int \alpha\phi^2 (|\text{Rm}|^2+\beta)^{\alpha-1}\square |\Rm|^2\,dV_t-\int 4\alpha (\alpha-1) \phi^2 (|\text{Rm}|^2+\beta)^{\alpha-2} |\Rm|^2 |\nabla |\Rm||^2\,dV_t \\&
-\int 8\alpha \phi \langle \nabla \phi, \nabla |\Rm|\rangle |\Rm|(|\Rm|^2+\beta)^{\alpha-1}\,dV_t= \textbf{I}+\textbf{II}+\textbf{III}+\textbf{IV}+\textbf{V}.
\end{align*}
Clearly $\textbf{II}\leq 0$ and since $\square \phi\leq 0$, we have $\textbf{I}\leq 0$ as well. Also, by the evolution of the curvature tensor and Kato's inequality $|\nabla |T||\leq |\nabla T|$, we have
\ba
\textbf{III}&\leq %-2\alpha\int\phi^2 (|\text{Rm}|^2+\beta)^{\alpha-1} |\nabla\Rm|^2\,dV_t+16\alpha\int \phi^2 (|\text{Rm}|^2+\beta)^{\alpha+\frac{1}{2}} \,dV_t\\&\leq 
%-2\alpha\int\phi^2 (|\text{Rm}|^2+\beta)^{\alpha-1} |\nabla|\Rm||^2\,dV_t+16\alpha\int \phi^2 (|\text{Rm}|^2+\beta)^{\alpha+\frac{1}{2}} \,dV_t\\&\leq
-2\alpha\int\phi^2 (|\text{Rm}|^2+\beta)^{\alpha-2} |\Rm|^2|\nabla|\Rm||^2\,dV_t+16\alpha\int \phi^2 (|\text{Rm}|^2+\beta)^{\alpha+\frac{1}{2}} \,dV_t=\textbf{IIIa}+\textbf{IIIb}.
\end{align*}
Applying Cauchy-Schwarz and Young's inequality gives
\ba
\textbf{V}&\leq %\int 8\alpha \phi |\nabla \phi| |\nabla |\Rm|| \,|\Rm|(|\Rm|^2+\beta)^{\alpha-1}\,dV_t\\&\leq 
%8\alpha \left(\int |\nabla \phi|^2 (|\text{Rm}|^2+\beta)^{\alpha}\,dV_t\right)^{\frac{1}{2}}\left(\int \phi^2 (|\text{Rm}|^2+\beta)^{\alpha-2} |\Rm|^2 |\nabla |\Rm||^2\,dV_t \right)^{\frac{1}{2}}\\&\leq 
\frac{4\alpha}{\ep}\int |\nabla \phi|^2 (|\text{Rm}|^2+\beta)^{\alpha}\,dV_t+4\alpha\ep \int \phi^2 (|\text{Rm}|^2+\beta)^{\alpha-2} |\Rm|^2 |\nabla |\Rm||^2\,dV_t=\textbf{Va}+\textbf{Vb}.
\end{align*}
Now gather up the terms \textbf{IV}, \textbf{IIIa}, \textbf{Vb}, and set $\ep=\frac{1}{8}$ so that the coefficient is $c(\alpha):=\frac{3}{2}\alpha+4\alpha (\alpha-1)$, which satisfies $-c(\alpha)\in [\frac{2}{3}\alpha^2, 4\alpha^2]$. Then using Cauchy-Schwarz again gives
\begin{align*}
c(\alpha)&\int  \phi^2 (|\text{Rm}|^2+\beta)^{\alpha-2}|\text{Rm}|^2|\nabla |\text{Rm}||^2 \,dV_t=
c(\alpha)\frac{1}{\alpha^2}\int  \phi^2 |\nabla (|\text{Rm}|^2+\beta)^{\frac{\alpha}{2}}|^2 \,dV_t\\&\leq
c(\alpha)\frac{1}{\alpha^2}\left(\frac{1}{2}\int  |\nabla \phi (|\text{Rm}|^2+\beta)^{\frac{\alpha}{2}}|^2 \,dV_t-\int  |\nabla \phi|^2 (|\text{Rm}|^2+\beta)^{\alpha}\,dV_t\right)\\&\leq 
-\frac{1}{3}\int  |\nabla \phi (|\text{Rm}|^2+\beta)^{\frac{\alpha}{2}}|^2 \,dV_t+4\int  |\nabla \phi|^2 (|\text{Rm}|^2+\beta)^{\alpha}\,dV_t.
\end{align*}
The result follows by gathering like terms and simplifying.
%So putting these pieces back together gives
%\begin{align*}
%\frac{d}{dt}\int \phi^2&(|\text{Rm}|^2+\beta)^{\alpha}\,dV_t\leq 
%-\frac{1}{3}\int  |\nabla \phi (|\text{Rm}|^2+\beta)^{\frac{\alpha}{2}}|^2 \,dV_t\\&\;\;\;+
%(32\alpha+4) \int |\nabla \phi|^2 (|\Rm|^2+\beta)^\alpha\,dV_t+(n+32\alpha)\int \phi^2(|\text{Rm}|^2+\beta)^{\alpha+\frac{1}{2}}\,dV_t.
%\end{align*}
\end{proof}

\vspace{5pt}

\section{A bounded curvature Ricci flow existence result}\label{sectionboundedcurv}

\vspace{5pt}

As part of the proof of Theorem \ref{main2}, we will require a short-time Ricci flow existence result for bounded curvature metrics with the properties that the flow enjoys $1/t$ curvature control and preservation of the local curvature concentration (see Theorem \ref{shorttimeexistence}). This result is similar to {\cite[Theorem 3.1]{CHL}}, which in turn is only a slight generalization of {\cite[Theorem 1.2]{CCL}}. Because of these similarities, we may omit some of the details in the proof of Theorem \ref{shorttimeexistence} when appropriate and refer the reader to those works. Before stating that however, we first prove an elementary lemma which states that the local curvature concentration can only grow at most at some uniform rate under bounded curvature flows.

\vspace{5pt}

\begin{lem}\label{uniformincreasingconcentrationlem}
Let $(M^n, g(t))$, $t\in [0,T]$, $n\geq 3$ be a complete bounded curvature Ricci flow such that
\[
\sup_{x\in M} \int_{B_{g(0)}(x,r_1)}|\Rm|_{g(0)}^{\frac{n}{2}}\,dV_{g(0)}\leq \sigma
\]
for some $r_1>0$, $\sigma>0$. Then for any $r_2<r_1$ and $\ep>0$, there exists $\eta>0$ such that for each $t\in [0,\eta]$, there holds 
\[
\sup_{x\in M} \int_{B_{g(t)}(x,r_2)}|\Rm|_{g(t)}^{\frac{n}{2}}\,dV_{g(t)}\leq \sigma+\ep.
\]
\end{lem}

\begin{proof}
For the sake of brevity, we write $K=\sup|\Rm(x,t)|+1$ where the supremum is taken over all $(x,t)\in M\times [0,T]$. In what follows, we will write $C$ to be some constant possibly depending on $n, r_1, r_2, \sigma, K$ which may vary line-by-line. Let $\psi: \R\to \R$ be some smooth decreasing cutoff function such that $\psi\equiv 1$ on $(-\infty, 0]$, $\psi\equiv 0$ on $[1,\infty)$ and $\psi'\geq -2$. Now for any $x\in M$, let $\phi_x: M \to \R$ be defined as
\[
\phi_x(y)=\psi\left(\frac{2}{r_1-r_2}d_{g(0)}(x, y)-\frac{r_1+r_2}{r_1-r_2}\right).
\]
Then $\supp\,\phi_x\subset B_{g(0)}(x, r_1)$, $\phi_x \equiv 1$ on $B_{g(0)}(x, \frac{r_1+r_2}{2})$, and $|\nabla^{g(0)} \phi_x|^2_{g(0)}\leq C$. By the fact that the flow has bounded curvature, we may assume that $\eta>0$ is small enough so that for all $(x,t)\in M\times [0,\eta]$, we have 
\[
\begin{cases}
&B_{g(t)}(x,  r_2)\subset B_{g(0)}(x, \frac{r_1+r_2}{2}), \\
&|\nabla^{g(t)} \phi_x|^2_{g(t)}\leq 2|\nabla^{g(0)} \phi_x|^2_{g(0)}, \text{ and } \\
&\Vol_{g(t)}B_{g(0)}(x,r_1)\leq C.
\end{cases}
\]
Now for any $0<\beta\leq K$, we apply {\cite[Lemma 4.1]{CM}} to see that 
\[
\frac{d}{dt}\int \phi^2 (|\Rm|_{g(t)}^2+\beta)^{\frac{n}{4}}\,dV_{g(t)}\leq C\int \phi^2 (|\Rm|_{g(t)}^2+\beta)^{\frac{n}{4}+\frac{1}{2}}\,dV_{g(t)}+C\int |\nabla^{g(t)} \phi|^2_{g(t)} (|\Rm|_{g(t)}^2+\beta)^{\frac{n}{4}}\,dV_{g(t)}\leq C.
\]
Integrating in time then sending $\beta\to 0$, we can apply our properties of $\phi$ to obtain
\ba
\int_{B_{g(t)}(x_0,  r_2)}|\Rm|_{g(t)}^{\frac{n}{2}}\,dV_{g(t)}&\leq \int \phi^2|\Rm|_{g(t)}^{\frac{n}{2}}\,dV_{g(t)}\leq \int \phi^2|\Rm|_{g(0)}^{\frac{n}{2}}\,dV_{g(0)}+Ct\\&\leq \int_{B_{g(0)}(x_0, r_1)} |\Rm|_{g(0)}^{\frac{n}{2}}\,dV_{g(0)}+Ct. 
\end{align*}
The result follows by restricting $\eta\leq\ep/ C$. 
\end{proof}

\vspace{5pt}

\begin{thm}\label{shorttimeexistence}
For all $n\geq 3$, $\lambda\geq 1$, $v_0>0$, there exists $C(n)\geq 1$, $c(n)>0$, $\iota(n,\lambda, v_0)$, $\delta(n, \lambda)>0$, $\Lambda(n, \lambda )> 1$, and $\widetilde T(n, \lambda, v_0)>0$ such that the following holds. Suppose that $(M^n, g(t))$, $t\in [0,T]$ is a complete bounded curvature Ricci flow and the initial metric $g(0)$ satisfies the following for all $x\in M$: 
\begin{enumerate}[(a)]
\item$\Rc_{g(0)}(x)\geq -1$;\label{shorttimepropscalar}
%\item $\sR_{g(0)}(x)\geq -1$;\label{shorttimepropscalar}
\item $\Vol_{g(0)} B_{g(0)}(x,2r)\leq 2^{2n+1} \cdot \Vol_{g(0)} B_{g(0)}(x,r)\leq 2^{2n+2} \omega_n r^n$ for all $r\leq 4$; and \label{shorttimepropdoubling}
\item $\int_{B_{g(0)}(x,4)}|\Rm|_{g(0)}^{\frac{n}{2}}\,dV_{g(0)}\leq \delta(n,\lambda) v_0$.\label{shorttimepropconcentration}
\end{enumerate}
Moreover for each $x\in M$ and $r\in (0, 1]$, assume that 
\[
v_{xr}:=\frac{\Vol_{g(0)} B_{g(0)}(x,r)}{r^n}\geq v_0
\]
and that
\be\label{shorttimepropnu}
\Sob(B_{g(0)}(x,4r),g(0))\leq \left[\lambda v_{xr}^{-\frac{2}{n}},0\right] .
\ee
Then the flow $g(t)$ satisfies each of the following estimates uniformly for all $(x,t)\in M\times [0,\widetilde T\land T]$:
\begin{enumerate}[(i)]
\item $\sR_{g(t)}(x)\geq  -n$;  \label{shorttimethmconc4}
\item $|\Rm|_{g(t)}(x)\leq  \frac{1}{t}$;\label{shorttimethmconc1}
\item $\inj_{g(t)}(x) \geq \iota(n, \lambda, v_0) \sqrt{t}$; \label{shorttimethmconc2}
\item $\Vol_{g(t)}B_{g(t)}(x, r)\geq c(n) \lambda^{-\frac{n}{2}} v_0\, r^n$ for all $r\in (0,1]$;\label{shorttimethmconc3}
\item $\int_{B_{g(t)}(x,1)}|\Rm|_{g(t)}^{\frac{n}{2}}\,dV_{g(t)}\leq \Lambda(n,\lambda)\delta(n,\lambda) v_0$; and  \label{shorttimethmconc5}
\item $\Sob(B_{g(t)}(x,r),g(t))\leq \left[C(n)\lambda v_{xr}^{-\frac{2}{n}},1\right]$ for all $r\in (0,1]$. \label{shorttimethmconc6}
\end{enumerate}
In particular, the flow exists up until time $\widetilde T$.
\end{thm}

\vspace{5pt}

\begin{remark}\label{scalarreplacericci}
The lower Ricci bound assumption \eqref{shorttimepropscalar} can be weakened to a lower bound on the scalar curvature $\sR_{g(0)}\geq -1$ in the precense of the volume doubling assumption \eqref{shorttimepropdoubling}, at the expense of possibly shrinking $\delta(n,\lambda)$. We choose to state the theorem like this because it is more natural in our setting, and because this stronger hypothesis allows us to substitute a limiting contradiction argument for the Moser Iteration argument (see Step 2 in the proof).
\end{remark}

\vspace{5pt}

\begin{proof}
Note that since lower scalar curvature bounds are preserved by complete bounded curvature flows, conclusion \eqref{shorttimethmconc4} is trivially satisfied (we have included it in the theorem statement for clarity later). We begin the proof of the other properties by fixing our constants. First, we set
\be\label{boundedcurvdefnofT}
\widetilde T :=\min\left\{T_1,C_4\inv, \frac{1}{16}\beta^{-2},\frac{1}{4}C_3\inv, \frac{\log 2}{10}, \frac{r_0^4}{10^4}, r_0^2 \widehat T\right\}.
\ee
Here $T_1=T_1(n,\lambda, v_0)>0$ is some positive time coming from Step 2 of the proof, $\beta(n)$ is from the Shrinking Balls Lemma \ref{SBL}, $\widehat T=\widehat T(n,1,1)$ is from the Expanding Balls Lemma \ref{EBL}, and $c_1(n)$, $C_3(n)$, $C_4(n,\lambda)$, $r_0(n,\lambda)$ are constants that will be defined within the proof. Also, we define
\be\label{smallnessdeltaassumptionfixed}
\delta(n,\lambda)= \Lambda(n,\lambda)\inv \left(\frac{1}{10^7 n^3 \lambda}\right)^{\frac{n}{2}},
\ee
where $\Lambda(n,\lambda)$ is defined in \eqref{boundonLambda}. We will see in Step 1 of the proof that we may take 
\[
C(n)= 1000n \;\;\text{ and } \;\; c(n) = (2^{n+4}\cdot 1000n)^{-\frac{n}{2}}.
\]
Finally, recall that condition \eqref{shorttimethmconc2} follows from the conclusions \eqref{shorttimethmconc1} and \eqref{shorttimethmconc3} by \cite{CGT}, so we will assume that the $\iota(n,\lambda, v_0)>0$ is the constant given there.\\

Now we may proceed with the proof. Let $T'$ be the maximal time in the interval $[0, T\land \widetilde T]$ such that for all $(x,t)\in M\times [0,T']$, there holds 
\be\label{locsmallconc}
\int_{B_{g(t)}(x,1)}|\Rm|_{g(t)}^{\frac{n}{2}}\,dV_{g(t)}\leq \Lambda(n,\lambda)\delta(n,\lambda) v_0.
\ee
Note that because $(M, g(t))$, $t\in [0,T]$ is a bounded curvature flow which satisfies \eqref{shorttimepropconcentration}, Lemma \ref{uniformincreasingconcentrationlem} guarantees that $T'>0$. In advance however, it appears as though $T'$ depends on the curvature bound of the flow $g(t)$, but we will show that this is not the case.  \\

Our proof is broken down into three steps: (1) Show we have a uniform local Sobolev inequality (i.e., conclusion \eqref{shorttimethmconc6}) along the flow on the interval $[0,T']$ assuming the time interval is possibly shortened so that a global $\alpha/t$ curvature bound holds, which in particular implies conclusion \eqref{shorttimethmconc3}; (2) Show that the assumption of the global $\alpha/t$ curvature bound assumption in Step 1 is implied as long as $T'\leq \widetilde T$ (this gives conclusions \eqref{shorttimethmconc1} and \eqref{shorttimethmconc2}); (3) Derive a contradiction in the event that $T'< T\land \widetilde{T}$. \\

\textbf{Step 1 (Sobolev inequality)}\\

Fix some $r_0(n,\lambda)\in (0,1)$ to be defined later (in hindsight, we could take $r_0=[2^{3n+6}(2^{n+4}1000n \lambda )^{\frac{n}{2}}]\inv$). Since $(M,g(t))$, $t\in [0,T]$ is a bounded curvature flow, we may assume that there is some $T_1>0$ such that 
\be\label{boundedcurvstep1curvbound}
|\Rm|_{g(t)}\leq \frac{3r_0/10}{t} \;\; \text{ for all } t\in (0,T_1].
\ee
In this step of the proof, we will restrict to the interval $[0,T_1\land T']$. We will show in Step 2 that it is necessarily true that $T_1$ can be taken to depend only on the variables $n,\lambda, v_0$. \\ 

Fix some $x\in M$ and $r\in [r_0, 1]$. By our smallness condition on $\delta(n,\lambda)$ in \eqref{smallnessdeltaassumptionfixed} and our assumptions \eqref{shorttimepropnu} and \eqref{shorttimepropconcentration}, as well as the fact that $v_{x,r}\geq v_0$, we may apply Lemma \ref{nufromSob} (with $a=0$) to obtain
\[
\nu(B_{g(0)}(x,4 r), g(0),2)\geq -\frac{n}{2}\log(C_1(n)\lambda v_{x,r}^{-\frac{2}{n}}).
\]
Now by our curvature assumption \eqref{boundedcurvstep1curvbound}, we apply the ``almost-monotonicity" of this local functional along the Ricci flow {\cite[Lemma 2.5]{Cheng}} (with the variables there being $A=\frac{3r}{10}, B=\frac{1}{4}$, $D=0$, $\tau_1=1$) to see that for each $y\in B_{g(0)}(x, r)$ and $t\in [0, T_1\land T']$ there holds 
\ba
\nu(B_{g(t)}(y, r), g(t), 1)&\geq \nu(B_{g(0)}(y, 3r), g(0), 2)-\left(\frac{2}{10(3r/10)^2 (1/4)^2}+e\inv\right)(e^{\frac{T'}{10(3r/10)^2 (1/4)^2}}-1)\\&
\geq \nu(B_{g(0)}(x, 4r), g(0), 2)-\frac{n}{2}\log 2
\geq -\frac{n}{2}\log(2C_1(n)\lambda v_{x,r}^{-\frac{2}{n}}).
\end{align*}
Note that in the second inequality, we have used the fact that $T'\leq \widetilde T\leq \frac{r_0^4}{10^4}$ so that 
\[
\left(\frac{2}{10(3r/10)^2 (1/4)^2}+e\inv\right)(e^{\frac{T'}{10(3r/10)^2 (1/4)^2}}-1)\leq \frac{100}{r^2}(e^{\frac{50  T'}{r^2}}-1)\leq 10^4 r_0^{-4} \widetilde T \leq 1< \frac{n}{2}\log 2.
\]
Since our smallness condition on $\delta$ in \eqref{smallnessdeltaassumptionfixed} in particular implies that 
\[
\delta(n,\lambda)\leq \Lambda(n,\lambda)\inv\left(\frac{1}{250 n\lambda}\right)^{\frac{n}{2}},
\]
then by Lemma \ref{Sobfromnu} and Remark \ref{productofconstantsbound}, we obtain (with $C_2(n)=1000n$) the uniform Sobolev inequality 
\be\label{shorttimesobineq}
\Sob(B_{g(t)}(y,r), g(t))\leq [C_2(n)\lambda v_{x,r}^{-\frac{2}{n}}, 1]
\ee
whenever $y\in B_{g(0)}(x,r)$, which gives conclusion \eqref{shorttimethmconc6}. Note that by Lemma \ref{volfromSob}, we can also say that 
\be\label{shorttimevolumestep1}
\Vol_{g(t)} B_{g(t)}(y,\rho) \geq c_1(n)\lambda^{-\frac{n}{2}} v_{x,r} \rho^n \;\;\text{ whenever } \;  y\in B_{g(0)}(x,r), \; \rho\in (0,r], \; t\in [0, T_1\land T'].
\ee
This gives conclusion \eqref{shorttimethmconc3} since $v_{x,r}\geq v_0$. Furthermore, note that $c_1(n)\geq (2^{n+4}\cdot C_2(n))^{-\frac{n}{2}}$. We also point out that although the work in this step was valid for any $r\in [r_0,1]$, we will only need to apply it with $r=r_0$ and $r=1$. This volume ratio estimate, together with the assumed curvature bound on $(0,T_1\land T']$ imply (by \cite{CGT}) the injectivity radius bound 
\be\label{shorttimeinjectivitystep1}
\inj_{g(t)}(x)\geq \iota(n,\lambda, v_0)\sqrt{t} \; \text{ for all } x\in M, \;t\in (0,T_1\land T'].
\ee
This gives conclusion \eqref{shorttimethmconc2}. \\

\textbf{Step 2 (Curvature bound)}\\

In this step of the proof, we will show that the flow $g(t)$ satisfies the curvature bound \eqref{boundedcurvstep1curvbound} as long as \eqref{locsmallconc} holds on the time interval $[0,T']$, assuming that $T'\leq T_1$ where $T_1$ depends only on $n,\lambda, v_0$. Just like the proof of {\cite[Proposition 2.1]{CCL}}, we could do this by applying Moser's iteration, which is well-known to experts but is a somewhat technical argument, especially in our setting where we must prove that the smallness of $\int |\Rm|^{\frac{n}{2}}$ need only depend linearly on $v$. Instead, we will argue by contradiction by employing a limiting contradiction argument with respect to the existence time. Thus we may assume that for some $n\geq 3$, $\lambda\geq 1$, $v_0>0$, there exists a sequence of times $T_i\to 0$ and associated bounded curvature Ricci flows $(M_i^n, g_i(t))$, $t\in [0, T_i]$ each satisfying the hypotheses of the theorem and such that 
\[
\int_{B_{g_i(t)}(x,1)}|\Rm|_{g_i(t)}^{\frac{n}{2}}\,dV_{g_i(t)} \leq \Lambda \delta v_0
\]
for all $(x,t)\in M_i\times [0,T_i]$, but which do not enjoy a $\alpha/t$ curvature bound on the interval $[0,T_i]$, where $\alpha=3r_0/10$. Since each $(M_i, g_i(t))$, $t\in [0, T_i]$ is assumed to be a bounded curvature flow, we may assume that there is some \emph{first time} $t_i\in [0,T_i]$ which fails our desired bound. In other words we assume that 
\[
|\Rm|_{g_i(t_i)}(x_i)= \frac{\alpha}{t_i}
\]
for some $x_i\in M_i$ and that 
\[
|\Rm|_{g_i(t)}\leq \frac{\alpha}{t} \; \text{ on } M_i\times (0, t_i]. 
\]
Moreover by Step 1, we know that the Sobolev inequality \eqref{shorttimesobineq} and volume ratio bound \eqref{shorttimevolumestep1} hold with $r=1$, $\rho=\sqrt{t}$, $v_{x,1}=v_0$. Recall this in turn gave the injectivity radius estimate \eqref{shorttimeinjectivitystep1} for each $g_i(t)$ on $(0,t_i]$.\\

 Now parabolically rescale the $g_i(t)$ so that $t_i\nearrow  1$, which is to say we define the sequence of complete flows $\tilde g_i(t)=\frac{1}{t_i}g_i(t_i t)$, each defined on $[0,1]$, which satisfy
\be\label{boundedcurvstep2props}
\begin{cases}
& \Rc_{\tilde g_i(0)}\geq -t_i,\; \; \sR_{\tilde g_i(t)}\geq -nt_i ,\\
&|\Rm|_{\tilde g_i(1)}(x_i)=\alpha,\\
&|\Rm|_{\tilde g_i(t)}\leq \frac{\alpha}{t} \; \text{ for all } t\in (0,1], \\
&\int_{\Omega_i}|\Rm|_{\tilde g_i(1)}^{\frac{n}{2}}\,dV_{\tilde g_i(1)} \leq \Lambda \delta v_0,\\ 
&\inj_{\tilde g_i(1)}(x_i)\geq \iota(n,\lambda, v_0), \text{ and }\\
&\Sob(\Omega_i, \tilde g(1))\leq [C_2(n)\lambda v_0^{-\frac{2}{n}}, t_i].
\end{cases}
\ee
Here we have denoted $\Omega_i=B_{\tilde g(1)}(x,t_i^{-\frac{1}{2}})$. Therefore we may apply Hamilton's Compactness Theorem \cite{Hamilton} and a diagonal argument to say that $(M_i, \tilde g_i(t), x_i)$ subconverges (in the pointed Cheeger-Gromov sense) to some complete Ricci flow $(N,g_\infty(t), x_\infty)$, $t\in (0,1]$, which by virtue of the properties \eqref{boundedcurvstep2props}, must satisfy
\[
\begin{cases}
&\sR_{g_\infty (t)}\geq 0 \;\text{ for all } t\in (0,1],\\
&|\Rm|_{g_\infty(1)}(x_\infty)=\alpha,\\
& \sup_{x\in N}|\Rm|_{g_\infty(t)}(x) \leq \frac{\alpha}{t}\;\text{ for all } t\in (0,1],\\
&\int_N|\Rm|_{g_\infty(1)}^{\frac{n}{2}}\,dV_{g_\infty(1)} \leq \Lambda \delta v_0, \;\text{ and }\\ 
&\Sob(N, g_\infty(1))\leq [C_2(n)\lambda v_0^{-\frac{2}{n}}, 0].
\end{cases}
\]
Because $C_2(n)=1000n$, our requirement on $\delta(n,\lambda)$ in \eqref{smallnessdeltaassumptionfixed} implies that
\be\label{smallerthanCMresult}
\delta(n,\lambda)\leq \Lambda(n,\lambda)\inv\left(\frac{\delta(n)}{C_2(n)\lambda}\right)^{\frac{n}{2}} 
\ee
where $\delta(n)\geq \frac{1}{10^4 n^2}$ is the dimensional constant from {\cite[Theorem 1.2]{CM}}. According to \cite{CM}, we may therefore extend the flow $g_\infty(t)$ to the interval $(0,\infty)$. Now by an identical blow-down argument as in the proof of Theorem \ref{main} from section \ref{proofOfMain}, along with the faster than $1/t$ curvature decay, we conclude that for each $t\in (0,1]$, there holds
\be\label{avralmostnonnegativericci}
\lim_{r\to \infty}\frac{\Vol_{g_\infty(t)}B_{g_\infty(t)}(x_\infty, r)}{r^n}\geq \omega_n.
\ee
Now we have the following claim. 
\begin{claim}
For any $A\geq 1$, there exists $N_A\in \N$ such that whenever $i\geq N_A$, there holds 
\[
\inf_{0<t\leq 1}\mu(B_{\tilde g_i(0)}(x_i, 20A\sqrt{t}), \tilde g_i(0), t)\geq -A^{-2}.
\]
\end{claim}
This claim will be sufficient to derive a contradiction because we may take $A=A(\alpha/2, n)$ from {\cite[Theorem 4.5]{WangB}}, which, in combination with \eqref{boundedcurvstep2props}, would yield the absurd inequality 
\[
\alpha=|\Rm|_{\tilde g_i(1)}(x_i)\leq \frac{\alpha}{2}.
\]
We will now prove the claim. 
\begin{proof}[Proof of Claim]
For any given $A\geq 1$ define $\eta=A^{-2}$ and let $\xi(\eta,n,A)>0$ be the constant from {\cite[Lemma 4.10]{WangB}}. It will suffice to show that for all $i$ sufficiently large, we in fact have
\be\label{almostEuclideanconditionRc}
\begin{cases}
&\Rc_{\tilde g_i(0)}(x)\geq -\xi \;\text{ for all } x\in M; \\
&\Vol_{\tilde g_i(0)}B_{\tilde g_i(0)}(x_i, r)\geq (1-\xi) \omega_n r^n\;\text{ for all } r\in (0,\xi].
\end{cases}
\ee
Indeed, for any $t\in (0,1]$, the metric $\frac{1}{t}\tilde g_i(0)$ will then also satisfy these conditions which implies 
\[
-A^{-2}\leq \mu\left(B_{\frac{1}{t}\tilde g_i(0)} (x_i, 20A), \frac{1}{t}\tilde g_i(0), 1\right)=\mu\left(B_{\tilde g_i(0)} (x_i, 20A\sqrt t), \tilde g_i(0),t\right)
\]
where we have applied {\cite[Lemma 4.10]{WangB}} in the first inequality while the second follows from the scaling properties of the $\mu$-functional.\\

First let $i$ be large enough to ensure that $t_i\leq \xi$, which shows the the lower Ricci bound condition. For the volume condition, we begin by defining $\hat\xi>0$ by $(1-\hat \xi)^4=(1-\xi)$. Fix $s\in (0,1]$ small enough so that 
\[
R^n \geq (1-\hat \xi) (R+\beta \sqrt{\alpha s})^n \;\text{ for all } R\geq \xi\inv,
\]
where $\beta=\beta(n)$ from Lemma \ref{SBL}. Now by the convergence $\tilde g_i(s)\to g_\infty(s)$ and the asymptotic volume estimate for $g_\infty(s)$ as in \eqref{avralmostnonnegativericci}, we can fix some $R_0\geq \xi\inv$ such that 
\[
\Vol_{\tilde g_i(s)}B_{\tilde g_i(s)}(x_i, R_0)\geq (1-\hat \xi) \omega_nR_0^n.
\]
We also assume that $i$ is large enough so that $e^{n s t_i}\leq 1+\hat \xi$. Finally by Bishop-Gromov Volume Comparison, we can say that for all $i$ large enough and $r\in (0,R_0+\beta\sqrt{\alpha s}]$, there holds
\[
\frac{\Vol_{\tilde g_i(0)}B_{\tilde g_i(0)}(x_i, r)}{r^n}\geq (1-\hat\xi)\frac{\Vol_{\tilde g_i(0)}B_{\tilde g_i(0)}(x_i, R_0+\beta\sqrt{\alpha s})}{(R_0+\beta\sqrt{\alpha s})^n}.
\]
Then for any $r\in (0,\xi\inv]$, these conditions on the largeness of $i$ imply
\ba
r^{-n}\Vol_{\tilde g_i(0)}B_{\tilde g_i(0)}(x_i, r)& \geq(1-\hat \xi)(R_0+\beta\sqrt{\alpha s})^{-n}\Vol_{\tilde g_i(0)}B_{\tilde g_i(0)}(x_0, R_0+\beta\sqrt{\alpha s})\\&\geq 
(1-\hat \xi)^2 (R_0+\beta\sqrt{\alpha s})^{-n}\Vol_{\tilde g_i(s)}B_{\tilde g_i(0)}(x_0, R_0+\beta\sqrt{\alpha s})\\&\geq 
(1-\hat \xi)^3 R_0^{-n}\Vol_{\tilde g_i(s)}B_{\tilde g_i(s)}(x_0, R_0)\\&\geq 
(1-\hat \xi)^4\omega_n =(1-\xi) \omega_n.
\end{align*}
This shows that condition \eqref{almostEuclideanconditionRc} holds and completes the proof of the claim, and thus of Step 2 as well.
\end{proof}

\vspace{5pt}

\textbf{Step 3 (Uniform time)}\\

We have already seen in Steps 1 and 2 that conclusions \eqref{shorttimethmconc4}, \eqref{shorttimethmconc1}, \eqref{shorttimethmconc2}, \eqref{shorttimethmconc3}, and \eqref{shorttimethmconc6} all follow as long as conclusion \eqref{shorttimethmconc5} holds on $[0,T']$ if $T'\leq T\land \widetilde T$. We will assume that $T'< T\land \widetilde T$ then show that this contradicts the maximality of $T'$ which will complete the proof of the theorem.  \\

 Start by fixing some $x_0\in M$ and let $\phi:M\to \R$ be defined as in \cite{CCL}. That is $\phi(x,t)=e^{-10t} \varphi(\eta(x,t))$ where $\eta(x,t)=d_{g(t)}(x,x_0)+C_3(n)\sqrt{t}$ and $\varphi: \R\to \R$ is a smooth cutoff function with $\varphi\equiv 1$ on $(-\infty, \frac{1}{2}]$, $\varphi\equiv 0$ on $[1,\infty)$, $\varphi''\geq -10 \varphi$, $0\geq \varphi'\geq -10\sqrt{\varphi}$. Moreover $C_3(n)$ is chosen large enough so that $\square \phi\leq 0$. With this choice of $\phi$ in Lemma \ref{curvconcevolutionlem}, we have
\[
\begin{split}
\frac{ d}{ dt}\int_M &\phi^2( |\Rm|_{g(t)}^2+\beta)^{\frac{n}{4}}\,dV_{g(t)}\leq -\frac{1}{3}\int_M |\nabla^{g(t)} (\phi(|\Rm|_{g(t)}^2+\beta)^{\frac{n}{8}})|_{g(t)}^2 \,dV_{g(t)}\\&+
5n\int_M \phi^2 (|\Rm|_{g(t)}^2+\beta)^{\frac{n}{4}+\frac{1}{2}}\,dV_{g(t)}+
1000n\int_{\supp \,\phi} (|\Rm|_{g(t)}^2+\beta)^{\frac{n}{4}}\,dV_{g(t)}.
\end{split}
\]
Now to the terms on the RHS, we apply our Sobolev inequality \eqref{shorttimesobineq} (here $r=1$ and recall that $v_{x,1}\geq v_0$) to the first and H\"older's inequality to the second to see that 
\be\label{boundedcurvthme11}
\begin{split}
\frac{ d}{ dt}\int_M \phi^2&( |\Rm|_{g(t)}^2+\beta)^{\frac{n}{4}}\,dV_{g(t)}\leq -\frac{1}{3C_2(n)\lambda v_0^{-\frac{2}{n}}}\left(\int_{B_{g(t)}(x,1)} |\phi (|\Rm|_{g(t)}^2+\beta)^{\frac{n}{8}}|^{\frac{2n}{n-2}}\,dV_{g(t)}\right)^{\frac{n-2}{n}}\\&+
5n\left(\int_{\supp \phi}  (|\Rm|_{g(t)}^2+\beta)^{\frac{n}{4}}\,dV_{g(t)}\right)^{\frac{2}{n}}\left(\int_{B_{g(t)}(x,1)} |\phi (|\Rm|_{g(t)}^2+\beta)^{\frac{n}{8}}|^{\frac{2n}{n-2}}\,dV_{g(t)}\right)^{\frac{n-2}{n}}\\&+
1001n\int_{\supp \phi} (|\Rm|_{g(t)}^2+\beta)^{\frac{n}{4}}\,dV_{g(t)}.
\end{split}
\ee
Then integrate over $[0,t]$ on both sides and send $\beta\to 0$. Because $\supp \,\phi\subset B_{g(t)}(x_0, 1)$, our restriction on $\delta(n,\lambda)$ in \eqref{smallnessdeltaassumptionfixed} ensures that the sum of the two terms containing $\|\phi^2 |\Rm|^{\frac{n}{2}}\|_{L^{\frac{n}{n-2}}}$ in \eqref{boundedcurvthme11} is nonpositive. Therefore we have
\[
\int_M \phi |\Rm|_{g(t)}^{\frac{n}{2}}\,dV_{g(t)}\leq \int_M \phi |\Rm|_{g(0)}^{\frac{n}{2}}\,dV_{g(0)}+1001n\Lambda \delta v_0 t\leq 
\delta v_0+1001n\Lambda \delta v_0 t=:(1+C_4 (n, \lambda) t )\delta v_0.
\]
Now we multiply through by $e^{10t}$ and notice whenever $t\leq T'\leq \widetilde T$, the conditions in \eqref{boundedcurvdefnofT} imply $e^{10t}\leq 2$, $C_4(n,\lambda)t\leq 1$, and  $\phi\equiv 1$ on $B_{g(t)}(x_0, 1/4)$. Then for all such $t\in [0, T']$, there holds
\be\label{boundedcurvsmallinttimet}
\int_{B_{g(t)}(x_0, \frac{1}{4})} |\Rm|^{\frac{n}{2}}\,dV_t\leq 4\delta v_0.
\ee

Now fix some $(y,t)\in M\times (0, T']$ and let $\{x_j\}\subset B_{g(t)}(y,3/2)$ be some subset of points such that $\{B_{g(t)}(x_j, \frac{1}{8})\}_{j=1}^N$ are mutually disjoint, but so that 
\be\label{coveringclaime1}
\bigsqcup_{i=1}^N B_{g(t)}(x_i, 1/8)\subset  B_{g(t)}(y,7/4) \; \; \text{ and } \; \; B_{g(t)}(y, 3/2)\subset \bigcup_{i=1}^N B_{g(t)}(x_i, 1/4).
\ee
Then by the Shrinking Balls Lemma \ref{SBL}, since $T'\leq \widetilde{T}\leq \frac{1}{16\beta^2}$, we have
\be\label{tin0containmentbounded}
B_{g(t)}(y,7/4) \subset B_{g(0)}(y,2).
\ee
We would like to be able to say that we have a similar reverse containment as well. For this, we will be enlisting the help of Lemma \ref{EBL}. Recall we have a $\alpha/t$ curvature bound by Step 2, a lower volume bound \eqref{shorttimevolumestep1} by Step 1, and an assumed lower bound on scalar curvature \eqref{shorttimepropscalar}. In summary, for all $x\in M$, there holds
\be\label{unscaledpropertiesboundedcurv}
\begin{cases}
&\sR_{g(0)}(x)\geq  -n; \\
&|\Rm|_{g(t)}(x)\leq \frac{\alpha}{t} \; \text{ for all } t\in [0,T']; \\
&\Vol_{g(0)} B_{g(0)}(x,r_0) =: v_{x,r_0} r_0^n; \; \text{ and } \\
&\Vol_{g(t)} B_{g(t)}(y,r) \geq c_1(n)\lambda^{-\frac{n}{2}} v_{x,r_0} r^n\; \text{ for all } y\in B_{g(0)}(x,r_0),  r\in (0,r_0],\; t\in [0,T'].
\end{cases}
\ee

Now we parabolically rescale $g(t)$ by definining
\[
\tilde g(t)=r_0^{-2} g\left(r_0^2 t\right)
\]
where $r_0(n,\lambda)=\frac{c_1(n)}{2^{3n+6} \lambda^{\frac{n}{2}}}$. Note that this rescaling has scaled distances $r_0\to 1$ and $\frac{1}{8}\to \frac{2^{3n+3} \lambda^{\frac{n}{2}}}{c_1(n)}$. Therefore, if we write $T_2=r_0^{-2}  T'$, we can see from the properties in \eqref{unscaledpropertiesboundedcurv}, and our volume doubling estimate \eqref{shorttimepropdoubling}, that the flow $\tilde g(t)$, in particular, satisfies 
\[
\begin{cases}
&\sR_{\tilde g(0)}(x)\geq -1;\\
&|\Rm|_{\tilde g(t)}(x)\leq t\inv \; \text{ for all } t\in [0, T_2]; \\
&\Vol_{\tilde g(0)} B_{\tilde g(0)}(x,2) \leq 2^{2n+1}v_{x,r_0} ;\; \text{ and } \\
&\Vol_{\tilde g(t)} B_{\tilde g(t)}(y,r) \geq c_1(n) \lambda^{-\frac{n}{2}} v_{x,r_0} r^n\; \text{ for all } y\in B_{\tilde g(0)}(x,1),  r\in (0,1/2],\; t\in [0, T_2].
\end{cases}
\]
Then Lemma \ref{EBL} implies that $B_{\tilde g(0)}(x, 1)\subset B_{\tilde g(t)}(x,2^{n+2}2^{2n+1} \lambda^{\frac{n}{2}} c_1(n)\inv )$ for all $x\in M$ and $t\in [0,T_2\land \widehat T(n,1,1)]$. Scaling back and recalling the last condition on the smallness of $\widetilde T$ in \eqref{boundedcurvdefnofT}, we can see that for all $i\in \{1, \dots, N\}$ there holds
\[
B_{g(0)}(x_i,r_0) \subset B_{g(t)}(x_i,1/8).
\]
Now starting from $r=r_0$, we apply our volume doubling estimate \eqref{shorttimepropdoubling} $k$ times, where $k(n,\lambda)$ is the unique integer so that so that
\[
4\leq 2^kr_0<8.
\]
Recalling that $x_i\in B_{g(t)}(y, 3/2)$ and the containment \eqref{tin0containmentbounded}, the doubling estimate yields
\[
2^{(2n+1) k}\Vol_{g(0)}B_{g(0)}(x_i,r_0) \geq \Vol_{g(0)}B_{g(0)}(x_i,2^k r_0) \geq \Vol_{g(0)}B_{g(0)}(y, 2)\geq \Vol_{g(0)}B_{g(t)}(y, 7/4).
\]
On the other hand, the $\{B_{g(0)}(x_i,r_0)\}_{i=1}^N$ are disjoint and each is contained within $B_{g(t)}(y,7/4)$ by \eqref{coveringclaime1}. Therefore
\[
N 2^{-(2n+1) k}\Vol_{g(0)}B_{g(t)}(y, 7/4)\leq \sum_{i=1}^N \Vol_{g(0)}B_{g(0)}(x_i,r_0)\leq \Vol_{g(0)}B_{g(t)}(y,7/4)
\]
from which we can readily conclude that $N\leq 2^{(2n+1) k}$. So if we define $\Lambda$ by
\be\label{boundonLambda}
\Lambda(n,\lambda):=8N\leq 8\left(\frac{2^{3n+9} \lambda^{\frac{n}{2}}}{c_1(n)}\right)^{2n+1}\leq 8\left(2^{3n+9} \lambda^{\frac{n}{2}}(2^{n+4} 1000n)^{\frac{n}{2}}\right)^{2n+1},
\ee
then we can combine \eqref{boundedcurvsmallinttimet} and \eqref{coveringclaime1} to see that for all $(y,t)\in M\times [0, T']$, there holds 
\be\label{shorttimee1234}
\int_{B_{g(t)}(y, 3/2)} |\Rm|^{\frac{n}{2}}\,dV_t\leq \sum_{i=1}^N \int_{B_{g(t)}(x_i, \frac{1}{4})} |\Rm|^{\frac{n}{2}}\,dV_t\leq 4N\delta v_0=\frac{1}{2}\Lambda \delta v_0.
\ee
But since $(M,g(t))$ is a bounded curvature flow, condition \eqref{shorttimee1234} (via an application of Lemma \ref{uniformincreasingconcentrationlem}) is in contradiction to the maximality of $T'$ when $T'<T\land \widetilde T$. This completes the proof. 
\end{proof}

\section{Proof of Theorem \ref{main2}}\label{sectionproofofmain}

\vspace{5pt}

Before beginning the proof of Theorem \ref{main2}, we need to recall some of the details from \cite{CHL}. The following is essentially taken out of the proof of their main theorem, with only one modification (as we will describe after).

\begin{thm}[{\cite[Proof of Theorem 1.1]{CHL}}]\label{CHLmainthm}
Let $(M^n, g_0, x_0)$ be some complete pointed noncompact Riemannian manifold of dimension $n\geq 4$ (not assumed bounded curvature), which exhibits a smooth function $\rho: M\to \R$ such that for some $C>0$ and all $x\in M$, there holds: 
\be\label{unboundedcurvdistancelikefunctionproperties}
\begin{cases}
&C\inv d_{g_0}(x_0,x)-C \leq \rho(x)\leq d_{g_0}(x_0, x) ;\\
& |\nabla \rho(x)|^2+|\Delta \rho(x)|\leq C; \text{ and }\\
&\int_{B_{g_0}(x,1)} |\nabla^2 \rho|^n \,dV_{g_0} \leq C .
\end{cases}
\ee
Then there exists a sequence sets of $U_j\subset M$ exhausting $M$, and metrics $g_j$ on $U_j$ such that: each $(U_j, g_j)$ is compete with bounded curvature; for any compact $\Omega\subset M$, $g_j|_\Omega=g_0|_\Omega$ for all $j$ large enough; and for each $x\in U_j$, the metric $g_j$ satisfy the locally equivalent condition
\be\label{unboundedcurvuniformlyequivalent}
c_{x,j}^2 g_0\leq g_j\leq C_{x,j}^2 g_0 \text{ on } B_{g_j}(x, 1) \text{ for some } C_{x,j}\geq c_{x,j}\geq 1 \text{ satisfying } \left(\frac{C_{x,j}}{c_{x,j}}\right)^n\leq 2. 
\ee
Moreover, if $\inf_{z\in M}\Rc_{g_0}(z)>-\infty$ then for any $\lambda, L,v_0, \alpha$, we may assume that each of the following implications holds (by possibly deleting finitely many of the $g_j$): 

\begin{enumerate}[(a)]
\item If $\Rc_{g_0}(x)\geq -\lambda$ for all $x\in M$, then $\Rc_{g_j}(x)\geq -2\lambda$ for all $x\in U_j$;  \label{CHLmainthmproof1}
\item If $\Vol_{g_0} B_{g_0}(x,2r)\leq L \cdot \Vol_{g_0} B_{g_0}(x,r)\leq L v_0 r^n$ for all $x\in M$ and $r\leq 1$, then  $\Vol_{g_j} B_{g_j}(x,2r)\leq 2L^2 \cdot \Vol_{g_j} B_{g_j}(x,r)\leq 4L^2 v_0 r^n$ for all $x\in U_j$, $r\leq 1/2$; and  \label{CHLmainthmproof2}
\item If $\int_{B_{g_0}(x,1)}|\Rm|_{g_0}^{\frac{n}{2}}\,dV_{g_0}\leq \alpha$ for all $x\in M$, then $\int_{B_{g_j}(x,1)}|\Rm|_{g_j}^{\frac{n}{2}}\,dV_{g_j}\leq 4\alpha$ for all $x\in U_j$. \label{CHLmainthmproof4}
\end{enumerate}
\end{thm}

\vspace{5pt}

\begin{proof}
Properties \eqref{CHLmainthmproof2} and \eqref{CHLmainthmproof4} are taken directly from {\cite[Claims 4.2, 4.4]{CHL}}, where we have assumed the $\kappa$ there is small enough so that $(1+c_2\kappa)^n\leq 2$, which is valid because $c_2$ was an absolute constant. So it remains to see that implication \eqref{CHLmainthmproof1} holds in our setting. Note that in \cite{CHL}, property \eqref{CHLmainthmproof1} is replaced with the synonymous implication on scalar curvature $\sR_{g_0}\geq -\lambda \Rightarrow \sR_{g_j}\geq -2\lambda$. We now describe how this modification is acheived. 

\vspace{5pt}

Using the notation from \cite{CHL} of $g_j=g_{R,0}=e^{2F_R}g_0$, where $F_R=\mathfrak{F}(\frac{\rho}{R})$ for $R$ sufficiently large (depending on $g_0$), we can compute the Ricci curvature of $g_{R,0}=g_j$ in normal coordinates $x^i$ at the point $p$ (with respect to $g_0$) by applying the conformal change formula:
\ba
&\Rc_{g_{R,0}}(\partial_i ,\partial_i )=\Rc_{g_0}(\partial_i ,\partial_i )-(n-2)\left(\nabla^{g_0}_i \nabla^{g_0}_i F_R-\nabla^{g_0}_iF_R\nabla^{g_0}_i F_R\right)-\left(\Delta^{g_0} F_R+(n-2) |dF_R|_{g_0}^2\right) g_0(\partial_i ,\partial_i )\\&\geq 
-\lambda g_0(\partial_i,\partial_i)-C(n)\left[\frac{|\mathfrak{F}''|}{R^2}\left|\frac{\partial \rho}{\partial x^i}\right|^2+\frac{|\mathfrak{F}'|}{R}\left|\frac{\partial^2 \rho}{\partial x^i\partial x^i}\right|+\frac{|\mathfrak{F}'|^2}{R^2}\left|\frac{\partial \rho}{\partial x^i}\right|^2\right]-\left(\Delta^{g_0} F_R+(n-2) |dF_R|_{g_0}^2\right) g_0(\partial_i ,\partial_i )\\&\geq 
-\lambda g_0(\partial_i,\partial_i)-C(n)\left[\frac{|\mathfrak{F}''|}{R^2}|\nabla^{g_0} \rho|^2+ \frac{|\mathfrak{F}'|}{R}|\Delta^{g_0} \rho|+\frac{|\mathfrak{F}'|^2}{R^2}|\nabla^{g_0} \rho|^2\right] g_0(\partial_i,\partial_i)\\&\geq 
-\lambda g_0(\partial_i,\partial_i)-\frac{C(n,g_0)}{R}e^{2F_R} g_0(\partial_i,\partial_i)=\left(-e^{-2F_R}\lambda -\frac{C(n,g_0)}{R}\right) e^{2F_R}g_{0}(\partial_i,\partial_i)\geq -2\lambda g_{R,0}(\partial_i,\partial_i).
\end{align*}
Here we have assumed $R\geq C(n,g_0) \lambda$, and have used that $\mathfrak{F}\geq 0$, $\text{exp}(-k\mathfrak{F})\mathfrak{F}^{(k)}$ is uniformly bounded for $k=1,2$ (by an absolute constant), and that $|\nabla^{g_0} \rho|$, $|\Delta^{g_0} \rho|$ are uniformly bounded (by a constant depending on $g_0$). We direct the reader to \cite{CHL} for more details. 
\end{proof}

\vspace{5pt}

We will be applying Theorem \ref{CHLmainthm} with the underlying manifold $(M,g_0)$ having nonnegative Ricci curvature and positive asymptotic volume ratio. In order to even get started, we clearly we must first demonstrate the existence of a distance-like function $\rho$ with the desired properties.

\vspace{5pt}

\begin{lem}\label{existenceofdistancefunctionlem}
Let $(M^n,g,x_0)$, $n\geq 4$ be a complete pointed Riemannian manifold with $\Rc_g\geq 0$, $\AVR(g)>0$, and $|\Rm|_g\in L^{\frac{n}{2}}(M,g)$. Then there exists a smooth $\rho: M\to \R$ and a constant $C>0$ such that condition \eqref{unboundedcurvdistancelikefunctionproperties} holds uniformly for all $x\in M$.\\
\end{lem}

\begin{proof}
First note by Corollary \ref{SobfromAVRcor} and Lemma \ref{nufromSob}, there holds 
\[
\bar\nu(M, g)\geq \log(C_1(n)\inv \AVR(g))=: -A>-\infty.
\]
Now since $|\Rm|_g\in L^{\frac{n}{2}}$, {\cite[Lemma 4.2]{CHL}} implies the existence of some $r\in (0,1]$ such that
\[
\sup_{x\in M }\int_{B_g(x,r)} |\Rm|_g^{\frac{n}{2}} \,dV_g \leq \delta(n, A, 1)^{\frac{n}{2}}
\]
where $\delta(n, A,1)$ is from {\cite[Theorem 2.1]{CHL}}. Then that result applied to the scaled up metric $\tilde g=r^{-2}g$ implies the existence of a smooth $\tilde g$-distance-like function $\tilde \rho$ which satisfies the properties in \eqref{unboundedcurvdistancelikefunctionproperties} relative to $\tilde g$ for some $C>0$. Then by taking $\rho=r\tilde \rho$ we can see that $\rho$ is a $g$-distance-like function satisfying \eqref{unboundedcurvdistancelikefunctionproperties} relative to $g$ for some (possibly different) $C>0$, albeit with the integral condition replaced by
\[
\int_{B_g(x,r)} |\nabla^2 \rho|^n \,dV_g \leq C.
\]
Now for any $x\in M$, let $\{x_i\}_{i=1}^N\subset B_g(x,1)$ be some set of points such that $\{B_g(x_i, r/2)\}_{i=1}^N$ are mutually disjoint but so that 
\[
\bigsqcup_{i=1}^N B_g(x_i, r/2)\subset B_g(x,2)\;\;\text{ and } \;\; B_g(x,1)\subset\bigcup_{i=1}^N B_g(x_i, r).
\]
Then by the assumption of $\Rc_g\geq 0$ and volume comparison, we have
\[
\omega_n 2^n\geq \Vol_g B_g(x,2)\geq \sum_{i=1}^N \Vol_g B_g(x_i, r/2)\geq N \omega_n \AVR(g) (r/2)^n.
\]
Therefore by the nature of the covering, there holds
\[
\int_{B(x,1)}|\nabla^2 \rho|^n \,dV_g \leq \sum_{i=1}^N \int_{B_g(x,r)} |\nabla^2 \rho|^n \,dV_g \leq  NC\leq 4^n r^{-n}\AVR(g)\inv C.
\]
This completes the proof of the lemma. 
\end{proof}

\vspace{5pt}

We therefore have the following approximation result with \emph{bounded curvature} metrics, from which we will readily apply our short-time existence result Theorem \ref{shorttimeexistence}. 

\vspace{5pt}

\begin{prop}\label{approximationunboundedthm}
For any $n\geq 4$ and $v>0$, there exists a dimensional constant $C(n)\geq 1$ such that the following holds. Let $(M^n, g)$ be a complete noncompact Riemannian manifold (not assumed bounded curvature) with $\Rc_g\geq 0$, $\AVR(g)\geq v$, and 
\[
\int_M|\Rm|^{\frac{n}{2}}_g\,dV_g\leq \alpha
\]
for some $\alpha>0$. Then there exists a sequence sets of $U_j\subset M$ exhausting $M$, and metrics $g_j$ on $U_j$ such that each $(U_j, g_j)$ is compete with bounded curvature and for any compact $\Omega\subset M$, $g_j|_\Omega=g|_\Omega$ for all $j$ large enough. Moreover each of the following conditions holds uniformly for every $x\in U_j$:
\begin{enumerate}[(i)]
\item $\Rc_{g_j}(x)\geq -1$;  \label{approximationunboundedthmcond1}
\item $\Vol_{g_j} B_{g_j}(x,2r)\leq 2^{2n+1}  \Vol_{g_j} B_{g_j}(x,r)\leq 2^{2n+2} \omega_n r^n$ for each $r\leq 4$;  \label{approximationunboundedthmcond2}
\item $\int_{B_{g_j}(x,4)}|\Rm|_{g_j}^{\frac{n}{2}}\,dV_{g_j}\leq 4\alpha$; \label{approximationunboundedthmcond3}
\item $\frac{\Vol_{g_j}B_{g_j}(x, r)}{r^n}\geq \frac{\omega_n}{2}v$ for each $r\leq 4$; and\label{approximationunboundedthmcond4a}
\item $\Sob(B_{g_j}(x,r),g_j)\leq \left[C(n)\left(\frac{\Vol_{g_j}B_{g_j}(x, r)}{r^n}\right)^{-\frac{2}{n}},0\right]$ for each $r\leq 4$.\label{approximationunboundedthmcond4b}
\end{enumerate}
\end{prop}

\begin{proof}
By the scale invariance of the hypotheses and each of the conclusions save \eqref{approximationunboundedthmcond1}, we may prove the theorem with $g_0=\frac{1}{32}g$ which allows us to only require that conditions \eqref{approximationunboundedthmcond2}, \eqref{approximationunboundedthmcond3}, \eqref{approximationunboundedthmcond4a}, \eqref{approximationunboundedthmcond4b} only hold for scales $r\leq 1/2$ (or larger). Clearly Lemma \ref{existenceofdistancefunctionlem} allows us to apply Theorem \ref{CHLmainthm}, which gives the existence of the $(U_j, g_j)$. Moreover, properties \eqref{approximationunboundedthmcond1}, \eqref{approximationunboundedthmcond2}, \eqref{approximationunboundedthmcond3} follow immediately from Theorem \ref{CHLmainthm} as well because $\Rc_{g_0}\geq 0$ and volume comparison implies
\[
\frac{\Vol_{g_0} B_{g_0}(x,2r)}{(2r)^n}\leq \frac{\Vol_{g_0} B_{g_0}(x,r)}{r^n}\leq \omega_n
\]
for all $r>0$. Now fix some $j\in \N$, $x\in U_j$, $r\in (0,1]$ and write $\frac{\Vol_{g_j}B_{g_j}(x, r)}{r^n}=v_{xjr}$. By the locally equivalent condition \eqref{unboundedcurvuniformlyequivalent} and the fact that $\AVR(g_0)\geq v$, we obtain
\[
v_{xjr}\geq c_{x,j}^n \frac{\Vol_{g_0}B_{g_j}(x, r)}{r^n}\geq c_{x,j}^n \frac{\Vol_{g_0}B_{g_0}(x, C_{x,j}\inv r)}{ r^n}\geq  \frac{1}{2}\frac{\Vol_{g_0}B_{g_0}(x, C_{x,j}\inv r)}{(C_{x,j}\inv r)^n}\geq \frac{ \omega_n}{2} v,
\]
which proves \eqref{approximationunboundedthmcond4a}. Now it remains to show \eqref{approximationunboundedthmcond4b}. Just as above, we may obtain an upper volume ratio bound:
\[
v_{xjr}\leq C_{x,j}^n \frac{\Vol_{g_0}B_{g_j}(x, r)}{r^n}\leq C_{x,j}^n \frac{\Vol_{g_0}B_{g_0}(x, c_{x,j}\inv r)}{ r^n}\leq  2 \frac{\Vol_{g_0}B_{g_0}(x, c_{x,j}\inv r)}{(c_{x,j}\inv r)^n}.
\]
Because $\Rc_{g_0}\geq 0$ and $C_{x,j}\inv r \leq c_{x,j}\inv r$, we may once again use volume comparison to say that 
\[
 \frac{1}{2}\frac{\Vol_{g_0}B_{g_0}(x, c_{x,j}\inv r)}{(c_{x,j}\inv r)^n}\leq v_{xjr}\leq 2\frac{\Vol_{g_0}B_{g_0}(x, c_{x,j}\inv r)}{(c_{x,j}\inv r)^n} \; \iff \; \frac{\Vol_{g_0}B_{g_0}(x, c_{x,j}\inv r)}{(c_{x,j}\inv r)^n}\in \left[\frac{1}{2}v_{xjr}, 2v_{xjr}\right].
\]
Now by Lemma \ref{SobfromAVR2}, we obtain the Sobolev inequality 
\[
\Sob(B_{g_0}(x,rc_{x,j}\inv), g_0)\leq \left[C_1(n) \left(\frac{v_{xjr}}{2}\right)^{-\frac{2}{n}},0\right]
\]
for some dimensional constant $C_1(n)\geq 1$. Since $B_{g_j}(x,r)\subset B_{g_0}(x,r c_{x,j}\inv)$, we may again apply the locally equivalent condition \eqref{unboundedcurvuniformlyequivalent} to obtain Sobolev inequality for $g_j$:
\ba
\left(\int_{B_{g_j}(x,r)}|u|^{\frac{2n}{n-2}}\,dV_{g_j}\right)^{\frac{n-2}{n}}&\leq 
\left(\frac{C_{x,j}}{c_{x,j}}\right)^{n-2} C_1(n) \left(\frac{v_{xjr}}{2}\right)^{-\frac{2}{n}} \int_{B_{g_j}(x,r)} |\nabla^{g_j} u|^2_{g_j} \,dV_{g_j}\\&\leq 
2 C_1(n) v_{xjr}^{-\frac{2}{n}} \int_{B_{g_j}(x,r)} |\nabla^{g_j} u|^2_{g_j} \,dV_{g_j}
\end{align*}
for all $u\in W_0^{1,2}(B_{g_j}(x,r))$. This gives \eqref{approximationunboundedthmcond4b} and completes the proof.  
\end{proof}

\vspace{5pt}

\begin{proof}[Proof of Theorem \ref{main2}]
Given $n\geq 4$ and $v\in (0,1]$, let $(M^n,g)$ satisfy the hypotheses of Theorem \ref{main} with this choice of $n,v$ and fix some $x_0\in M$. For any $R\geq 1$, denote $g_R=R^{-2}g$, which still satisfies $\AVR(g_R)\geq v$ and 
\[
\int_M |\Rm|_{g_R}^{\frac{n}{2}}\,dV_{g_R}\leq \ep(n) v.
\]
Then by Proposition \ref{approximationunboundedthm}, we obtain a sequence of sets $U_{R,j}\subset  M$ exhausting $M$ and complete bounded curvature metrics $g_{R,j}$ on $U_{R,j}$ which satisfy $g_{R,j}=g_R$ on $B_{g_R}(x_0, 1)$. Moreover each of the following conditions holds uniformly for every $x\in U_{R,j}$:
\begin{enumerate}[(i)]
\item $\Rc_{g_{R,j}}(x)\geq -1$;  
%\item $\sR_{g_{R,j}}(x)\geq -1$;  
\item $\Vol_{g_{R,j}} B_{g_{R,j}}(x,2r)\leq 2^{2n+1}  \Vol_{g_{R,j}} B_{g_{R,j}}(x,r)\leq 2^{2n+2} \omega_n r^n$ for each $r\leq 4$;  
\item $\int_{B_{g_{R,j}}(x,4)}|\Rm|_{g_{R,j}}^{\frac{n}{2}}\,dV_{g_{R,j}}\leq 4\ep(n) v$;
\item $\frac{\Vol_{g_{R,j}}B_{g_{R,j}}(x, r)}{r^n}\geq \frac{\omega_n}{2}v$ for each $r\in (0,1]$; and\label{proofofmain2cond4}
\item $\Sob(B_{g_j}(x,r),g_j)\leq \left[C_1(n)\left(\frac{\Vol_{g_j}B_{g_j}(x, r)}{r^n}\right)^{-\frac{2}{n}},0\right]$ for each $r\in (0,4]$\label{proofofmain2cond5}.
\end{enumerate}
Writing $v_0=\frac{\omega_n}{2}v$, we replace condition \eqref{proofofmain2cond4} with 
\[
v_{xRjr}:=\frac{\Vol_{g_{R,j}}B_{g_{R,j}}(x, r)}{r^n}\geq v_0 \;\text{ for each } r\in (0,1]\tag{iv'}, 
\]
and replace condition \eqref{proofofmain2cond5} with 
\[
\Sob(B_{g_j}(x,4r),g_j)\leq \left[4^{2}C_1(n)v_{xRjr}^{-\frac{2}{n}},0\right] \;\text{ for each } r\in (0,1]\tag{v'}.
\]
If we write $\lambda=4^2 C_1(n)$ and define
\be\label{epsilonboundedintermsofdelta}
\ep(n)= \frac{\delta(n,\lambda) \omega_n}{8},
\ee
then we may apply Theorem \ref{shorttimeexistence} to each metric $g_{R,j}$. Taking the diagonal scale normalized sequence $g_j=j^2 g_{j,j}$, the result is a sequence of flows $g_j(t)$, $t\in [0, j^2\widetilde T]$ such that $g_j(0)=g$ on $B_g(x_0,j)$ and for each $j\in\N$, all of the following are satisfied for every $(x,t)\in M\times [0,j^2\widetilde T]$:
\begin{enumerate}[(i)]
\item  $\sR_{g_j(t)}(x)\geq -nj^{-2} $;  
\item $|\Rm|_{g_j(t)}(x)\leq  \frac{1}{t}$;
\item $\inj_{g_j(t)}(x) \geq \iota(n,v) \sqrt{t}$; 
\item $\Vol_{g_j(t)}B_{g_j(t)}(x, r)\geq c_1(n) v\, r^n$ for all $r\in (0,j]$;
\item $\int_{B_{g_j(t)}(x,j)}|\Rm|_{g_j(t)}^{\frac{n}{2}}\,dV_{g_j(t)}\leq \Lambda(n,\lambda)\delta(n,\lambda) v_0$; and  
\item $\Sob(B_{g_j(t)}(x,j), g_j(t))\leq [C_2(n) \lambda v_0^{-\frac{2}{n}} ,\frac{1}{j^2}]$.
\end{enumerate}
Then just as in the proof of {\cite[Theorem 1.1]{CHL}}, we may apply {\cite[Corollary 3.2]{BingLongChen}}, Shi's modified interior estimates {\cite[Theorem 14.16]{Chowetal}}, the Arzel\`a-Ascoli Theorem, and a diagonal argument, we may conclude that some subsequence of $g_j(t)$ converges in $C^\infty_{loc}(M\times [0,\infty))$ to a long-time Ricci flow $g(t)$ with $g(0)=g$ satisfying the following for all $t\in [0,\infty)$:
\begin{enumerate}[(i)]
\item $\sR_{g(t)}\geq 0 $; 
\item $|\Rm|_{g(t)}\leq  \frac{1}{t}$;
\item $\inj_{g(t)} \geq \iota(n,v) \sqrt{t}$; 
\item $\Vol_{g(t)}B_{g(t)}(x_0, r)\geq c_1(n) v\, r^n$ for all $r\in (0,\infty)$;
\item $\int_M |\Rm|_{g(t)}^{\frac{n}{2}}\,dV_{g(t)}\leq \Lambda(n,\lambda)\delta(n,\lambda) v_0$; and  
\item $\Sob(M, g(t))\leq [C_2(n) \lambda v_0^{-\frac{2}{n}} ,0]$.
\end{enumerate}
Finally, obtaining the faster than $1/t$ curvature decay follows by applying {\cite[Theorem 1.2]{CM}} to the bounded curvature metric $g(1)$, which is validated by our smallness assumption on $\delta(n,\lambda)$ in \eqref{smallerthanCMresult}. This completes the proof of Theorem \ref{main2}. 
\end{proof}

\vspace{5pt}

\begin{remark}
By recalling our definitions of $\ep(n)$ in \eqref{epsilonboundedintermsofdelta}, of $\delta(n,\lambda)$ in \eqref{smallnessdeltaassumptionfixed}, and of $\Lambda(n,\lambda)$ in \eqref{boundonLambda}, as well as the fact that $\lambda\leq 32 (10n)^{3n}$ by Remark \ref{boundonSobfromvolconstant}, we obtain the explicit bound
\[
\ep(n)\geq \frac{\omega_n}{8}\left(\frac{1}{8\left(2^{3n+9} [32 (10n)^{3n}]^{\frac{n}{2}}(2^{n+4} 1000n)^{\frac{n}{2}}\right)^{2n+1}}\right) \left(\frac{1}{10^7 n^3 [32 (10n)^{3n}]}\right)^{\frac{n}{2}}\geq (10n)^{-10n^3}.
\]
\end{remark}

\vspace{5pt}

\textbf{Conflict of Interest:} The author declares that he has no conflict of interest.\\

\textbf{Ethics Approval:} The author declares that he has acted in accordance with the journal's ethical guidelines. \\

\textbf{Funding:} The author is partially supported by 4YF For PhD Students \#6456 at UBC.\\

\textbf{Data Availability:} Data sharing is not applicable to this article as no datasets were generated or analysed during
the current study.


\begin{thebibliography}{}


\bibitem{Brendle}
Brendle, Simon.
\newblock Sobolev inequalities in manifolds with nonnegative curvature.
\newblock {\em Comm. Pure Appl. Math.} \textbf{76} (2023), no. 9, 2192-2218.

\bibitem{CCL}
Chan Pak-Yeung; Chen, Eric; Lee, Man-Chun.
\newblock {Small curvature concentration and Ricci flow smoothing}.
\newblock {\em J. Funct. Anal.} \textbf{282} (2022), no. 10, Paper no. 109420, 29pp.

\bibitem{CHL}
Chan, Pak-Yeung; Huang, Shaochuang; Lee, Man-Chun.
\newblock {Manifolds with small curvature concentration}.
\newblock {\em Ann. PDE} \textbf{10} (2024), no. 2, Paper no. 23, 31pp.

\bibitem{ChanLee}
Chan, Pak-Yeung; Lee, Man-Chun.
\newblock {Gap theorem on manifolds with small curvature concentration}.
\newblock{ \em preprint} arXiv:2312.07845v1.

\bibitem{CM}
Chau, Albert; Martens, Adam.
\newblock {Long-time Ricci flow existence and topological rigidity from manifolds with pinched scale-invariant integral curvature}.
\newblock{ \em preprint} arXiv:2403.02564v1.

\bibitem{Cheeger}
Cheeger, Jeff.
\newblock Integral bounds on curvature elliptic estimates and rectifiability of singular sets.
\newblock {\em Geom. Funct. Anal.} \textbf{13} (2003), no. 1, 20–72.

\bibitem{CGT}
Cheeger, Jeff; Gromov, Mikhail; Taylor, Michael.
\newblock Finite propagation speed, kernel estimates for functions of the Laplace operator, and the geometry of complete Riemannian manifolds.
\newblock {\em J. Differential Geometry} \textbf{17} (1982), no. 1, 15-53.

\bibitem{BingLongChen}
Chen, Bing-Long.
\newblock {Strong uniqueness of the Ricci flow.}
\newblock {\em J. Differential Geom.} \textbf{82} (2009), no. 2, 363-382.

\bibitem{ChenEric}
Chen, Eric.
\newblock {Convergence of the Ricci flow on asymptotically flat manifolds with integral curvature pinching.}
\newblock {\em Ann. Sc. Norm. Super. Pisa Cl. Sci.} (5) \textbf{23} (2022), no. 1, 15-48.

\bibitem{Cheng}
Cheng, Liang.
\newblock {Pseudolocality theorems of Ricci flows on incomplete manifolds}.
\newblock{ \em preprint} arXiv:2210.15397v5.

\bibitem{Chowetal}
Chow, Bennett; Chu, Sun-Chin; Glickenstein, David; Guenther, Christine; Isenberg, James; Ivey, Tom; Knopf, Dan; Lu, Peng; Luo, Feng; Ni, Lei.
\newblock {The Ricci flow: techniques and applications. Part II.}
\newblock {\em American Mathematical Society, Providence, RI,}  2008. xxvi+458 pp.

\bibitem{Colding}
Colding, Tobias H.
\newblock {Ricci curvature and volume convergence.}
\newblock {\em Ann. of math.} \textbf{145} (1997), no. 3, 477-501.

\bibitem{Hamilton1}
Hamilton, Richard S.
\newblock Three-manifolds with positive Ricci curvature.
\newblock {\em J. Differential Geometry} \textbf{17} (1982), no. 2, 255-306.

\bibitem{Hamilton}
Hamilton, Richard S.
\newblock A compactness property for solutions of the Ricci flow.
\newblock {\em Amer. J. Math.} \textbf{117} (1995), no. 3, 545-572.

\bibitem{He}
He, Fai.
\newblock {Existence and applications of Ricci flows via pseudolocality}.
\newblock{ \em preprint} arXiv:1610.01735v2.

\bibitem{Li}
Li, Peter. 
\newblock {\em Geometric Analysis}, 
\newblock {\em Cambridge Studies in Advanced Mathematics,} Cambridge University Press, 2012.

\bibitem{Lee}
Lee, John.
\newblock {\em Introduction to Riemannian Manifolds}, volume 176 of 
\newblock {\em Grad. Texts in Math.,} 176 Springer, Cham, 2018. xiii+437 pp.

\bibitem{LeeTam}
Lee, Man-Chun; Tam, Luen-Fai.
\newblock Some local maximum principles along Ricci flows.
\newblock {\em Canad. J. Math.} \textbf{74} (2022), no. 2, 329-348.

\bibitem{ST}
Simon, Miles; Topping, Peter M.
\newblock {Local control on the geometry in 3D Ricci flow.}
\newblock {\em J. Differential Geom.} \textbf{122} (2022), no. 3, 467-518.

\bibitem{WangA}
Wang, Bing.
\newblock {The local entropy along Ricci flow Part A: the no-local-collapsing theorems.}
\newblock {\em Camb. J. Math.} \textbf{6} (2018), no. 3, 267-346.

\bibitem{WangB}
Wang, Bing.
\newblock {The local entropy along Ricci flow---Part B: the pseudolocality theorems.}
\newblock{ \em preprint} arXiv:2010.09981v1.

\bibitem{Ye}
Ye, Rugang.
\newblock The logarithmic Sobolev and Sobolev inequalities along the Ricci flow.
\newblock {\em Commun. Math. Stat.} \textbf{3} (2015), no. 1, 1-36

\end{thebibliography}
\end{document}